\documentclass[11pt]{article} 
\usepackage{amsmath,amssymb,amsthm,fullpage}
\usepackage{graphicx}

\theoremstyle{plain}
\newtheorem{theorem}{Theorem}
\newtheorem{lemma}{Lemma}
\newtheorem{proposition}{Proposition}
\newtheorem{corollary}{Corollary}
\newtheorem*{maintheorem}{Main theorem}
\theoremstyle{definition}

\newtheorem{example}{Example}

\newtheorem{definition}{Definition}
\newtheorem{remark}{Remark}

\newcommand{\al}{\alpha}

\newcommand{\de}{\delta}
\newcommand{\De}{\Delta}
\newcommand{\la}{\lambda}

\newcommand{\te}{\theta}

\newcommand{\eps}{\varepsilon}

\newcommand{\F}{\mathbb F}

\newcommand{\R}{\mathbb R}
\newcommand{\FF}{\mathcal F}
\newcommand{\PP}{\mathcal P}
\newcommand{\Sp}{{\mathbb S}^{n-1}}	

\newcommand{\brR}[1]{\!\left(#1\right)}	
\newcommand{\brS}[1]{\left[#1\right]}	

\newcommand{\brA}[1]{\left\langle#1\right\rangle}	

\newcommand{\set}[1]{\left\lbrace#1\right\rbrace}	
\newcommand{\Set}[1]{\Big\lbrace#1\Big\rbrace}		
\newcommand{\abs}[1]{\left|#1\right|}			
\newcommand{\Abs}[1]{\big|#1\big|}		
\newcommand{\norm}[1]{\left\|#1\right\|}	

\newcommand{\conv}{{\rm conv}}	
\newcommand{\aff}{{\rm aff}}		
\newcommand{\ext}{{\rm ext}}	
\newcommand{\supp}{{\rm supp}}
\newcommand{\spn}{{\rm span}}	

\newcommand{\dH}{{\rm d_\mathcal{H}}}
\newcommand{\dBM}{{\rm d_\mathcal{BM}}}
\newcommand{\sub}[1]{{^{}_{\!#1}}}		

\newcommand{\hide}[1]{}

\begin{document}

\title
{
{\bf Minimal volume product near Hanner polytopes}
\footnotetext{2010 Mathematics Subject Classification 52A20, 52A40, 52B11.}
}
\author{Jaegil Kim\thanks{Supported in part by U.S.~National Science Foundation grant DMS-1101636.}
}

\maketitle
\date{}

\begin{abstract}
Mahler's conjecture asks whether the cube is a minimizer for the volume product of a body and its polar in the class of symmetric convex bodies in a fixed dimension. It is known that every Hanner polytope has the same volume product as the cube or the cross-polytope. In this paper we prove that every Hanner polytope is a strict local minimizer for the volume product in the class of symmetric convex bodies endowed with the Banach-Mazur distance.
\end{abstract}

\section{Introduction and Notation}

A \emph{body} is a compact set which is the closure of its interior and, in particular, a \emph{convex body} in $\R^n$ is a compact convex set with nonempty interior. Let $K$ be a convex body in $\R^n$ containing the origin. Then the {\it polar} of $K$ is defined by $$K^\circ = \set{y\in\R^n \big| \brA{x,y}\le 1 \mbox{\ for all\ } x\in K}$$ where $\brA{\cdot,\cdot}$ is the usual scalar product in $\R^n$. The \emph{volume product} of a convex body $K$ is defined by $\PP(K)=\min\limits_{z\in K}\abs{K}\abs{(K-z)^\circ}$ where $\abs{\,\cdot\,}$ denotes the volume. In particular, if a convex body $K$ is {\it symmetric}, that is, $K=-K$, then  the volume product of $K$ is given by  $$\PP(K)=\abs{K}\abs{K^\circ}$$ because the minimum in the definition of the volume product is attained at the origin in this case. It turns out that the volume product is invariant under any invertible affine transformation on $\R^n$, and also invariant under the polarity, that is, for any invertible affine transformation $T:\R^n\rightarrow\R^n$, 
\begin{equation}\label{eq:affine_invariance}
\PP(TK)=\PP(K),\quad\text{and}\quad\PP(K^\circ)=\PP(K).
\end{equation}
Moreover, the volume product of the $\ell_p$-sum of two symmetric convex bodies is the same as that of the $\ell_q$-sum for $p,q\in[1,\infty]$ with $1/p+1/q=1$ (see \cite{SR}, \cite{Ku2}), that is, 
\begin{equation}\label{lp_invariance}
\PP(K\oplus_p L)=\PP(K\oplus_q L).
\end{equation}
We provide the definition and other details for the $\ell_p$-sum in Section \ref{direct_sums}.

It follows from F. John's theorem \cite{Jo} and the continuity of the volume function in the Banach-Mazur compactum that the volume product attains its maximum and minimum. It turns out that the maximum of the volume product is attained at the Euclidean ball $B_2^n$ (more generally, the ellipsoids). The corresponding inequality is known as the Blaschke-Santal\'o inequality \cite{Sa,Pe,MP}: for every convex body $K$ in $\R^n$,  
\begin{equation*}\tag{Blaschke-Santal\'o inequality}
\PP(K) \le \PP(B^n_2)\quad
\end{equation*}
with equality if and only if $K$ is an ellipsoid. For the minimum of the volume product, it was conjectured by Mahler in \cite{Ma,Ma2} that $\PP(K)$ is minimized at the cube $B_\infty^n$ in the class of symmetric convex bodies in $\R^n$, and minimized at the simplex $\De^n$ in the class of general convex bodies in $\R^n$. In other words, the conjecture asks whether the following inequalities are true: for any symmetric convex body $K$ in $\R^n$, 
\begin{equation*}\tag{symmetric Mahler's conjecture}
\PP(K) \ge \PP(B_\infty^n)
\end{equation*}
and $\PP(K) \ge \PP(\Delta^n)$ for any convex body $K$ in $\R^n$. The case of dimension $2$ was proved by Mahler \cite{Ma}. The Mahler's conjecture is affirmative for several special cases, like, e.g., absolutely symmetric bodies \cite{SR,Me,Re2}, zonoids \cite{Re,GMR}, and bodies of revolution \cite{MR}. An isomorphic version of the conjecture was proved by Bourgain and Milman \cite{BM}: there is a universal constant $c>0$ such that $\PP(K) \ge c^n\PP(B^n_2)$; see also different proofs by Kuperberg \cite{Ku}, Nazarov \cite{Na}, and Giannopoulos, Paouris, Vritsiou \cite{GPV}. Functional versions of the Blaschke-Santal\'o inequality and the Mahler's conjecture in terms of log-concave functions were investigated by Ball \cite{Ba}, Artstein, Klartag, Milman \cite{AKM}, and Fradelizi, Meyer \cite{FM,FM2,FM3}. For more information, see an expository article \cite{Ta} by Tao.

A symmetric convex body $H$ is called a {\em Hanner polytope} if $H$ is one-dimensional, or it is the $\ell_1$ or $\ell_\infty$ sum of two (lower dimensional) Hanner polytopes. It follows from \eqref{eq:affine_invariance}, \eqref{lp_invariance} that the volume product of the cube in $\R^n$ is the same as that of all Hanner polytopes in $\R^n$. Thus every Hanner polytope is also a candidate for a minimizer of the volume product among symmetric convex bodies. In this paper, we prove that every Hanner polytope is a local minimizer of the volume product in the symmetric setting. For the local behavior of the volume product, it is natural to consider the {\it Banach-Mazur distance} between symmetric convex bodies, which is defined by $$\dBM(K,L) = \inf \Set{c\ge1 : L\subset TK\subset cL, T\in{\rm GL}(n)},$$ because the volume product is invariant under linear transformations and the polarity. Due to such invariance properties of the volume product, we may fix a position of a convex body by taking a linear transformation. In this situation, we consider the {\it Hausdorff distance} $\dH$ which is defined by $$\dH(K,L) = \max\Big(\max_{x\in K}\min_{y\in L}|x-y|,\,\,\max_{y\in L}\min_{x\in K}|x-y|\Big).$$ In this paper we prove the following result:

\begin{maintheorem}
Let $K\subset\R^n$ be a symmetric convex body close enough to one of Hanner polytopes in the sense that $$\de=\min\Set{\dBM(K,H)-1: \text{$H$ is a Hanner polytope in $\R^n$}}$$ is small enough. Then $$\PP(K)\ge\PP(B_\infty^n)+c(n)\de$$ where $c(n)>0$ is a constant depending on the dimension $n$ only.
\end{maintheorem}

The local minimality of the volume product was first investigated in \cite{NPRZ} by Nazarov, Petrov, Ryabogin, and Zvavitch. Namely, they proved that the cube is a strict local minimizer of the volume product in the class of symmetric convex bodies endowed with the Banach-Mazur distance. It turns out that the basic procedure of the proof used in \cite{NPRZ} can be applied for other polytopes such as the simplex and the Hanner polytopes. In case of the simplex, the technique of \cite{NPRZ} can be adapted to the non-symmetric setting by showing the stability of the Santal\'o point, which leads to the local minimality of the simplex in non-symmetric setting by Reisner and the author \cite{KR}.
  
In case of Hanner polytopes, however, it is not so simple to get the same conclusion as the cube because the structure of a Hanner polytope may be much more complicated than that of the cube. To illustrate the structure of Hanner polytopes, we notice that Hanner polytopes are in one-to-one correspondence with the (perfect) graphs which do not contain any induced path of edge length $3$ (see Section \ref{delta_gap}). In this correspondence, the $n$-dimensional cube is associated with the graph of $n$ vertices without any edges. Thus, a Hanner polytope may have more delicate combinatorial structure that the cube, especially in large dimension. To overcome such difficulties, we employ and analyze the combinatorial representation of Hanner polytopes in Section \ref{delta_gap}.

In Section \ref{general_construction}, we investigate the technique of \cite{NPRZ} developed for the cube case, and then restate or generalize the key steps to get sufficient conditions for the local minimality of the Hanner polytopes. We provide the definition and basic properties of the $\ell_p$-sum of two polytopes in Section \ref{direct_sums}, which help us to characterize the faces and the flags of the $\ell_1$ or $\ell_\infty$-sum in terms of those of summands.  The sufficient conditions obtained from Section \ref{general_construction} will be verified for the Hanner polytopes in Sections \ref{Hanner_polytopes}, \ref{differential_volume_function}, \ref{delta_gap}. In particular, the combinatorial representation of vertices of a Hanner polytope will be useful for not only choosing a special position of a given convex body which is contained in the cube and contains the cross-polytope, but also finding a section or a projection of the convex body which makes the local minimality problem reduced to a lower dimensional situation. 

\hide{From a specific construction for Hanner polytopes given in Section \ref{Hanner_polytopes}, we prove intermediate steps in Sections \ref{differential_volume_function}, \ref{delta_gap}. These intermediate steps are very essential for the proof of Main theorem which is given in the end of Section \ref{delta_gap}. }

\medskip
We finish the introduction with the list of the notation used in the paper.

\medskip
\noindent{\bf Notation.} \quad Let $A$ be a non-empty subset of $\R^n$. Denote by $\brA{\cdot,\cdot}$ the usual inner product.
\begin{enumerate}
\item $\conv(A)$ : the minimal convex set containing $A$ \hfill ({\it convex hull})\qquad
\item $\spn(A)$ : the minimal linear subspace containing $A$ \hfill ({\it linear span})
\item $\aff(A)$ : the minimal affine subspace containing $A$ \hfill ({\it affine hull})
\item ${\rm int} (A) = \set{a\in\R^n:(a+\eps B_2^n)\cap\aff(A)\subset A \text{ for some }\eps>0}$\hfill ({\it interior})
\item $A^\circ\, = \set{b\in\R^n:\brA{a,b}\le1, \forall a\in A}$\hfill ({\it polar})
\item $A^\perp = \set{b\in\R^n:\brA{a,b}=0, \forall a\in A}$\hfill ({\it orthogonal complement})
\item $A-A=\set{a-b\in\R^n: a,b\in A}$
\item $\dim (A)$ : the dimension of $\spn(A-A)$ \hfill ({\it dimension})
\item ${\rm diam} (A) = \sup\set{|a|:a\in A-A}$\hfill ({\it diameter})
\end{enumerate}
In addition, by $c$ we denote a constant depending on dimension only, which may change from line to line, and write $f(\eps)=O(\eps)$ if $\sup\limits_{0<\eps<\eps_0}\abs{f(\eps)\over\eps}<\infty$ for some $\eps_0>0$, and $f(\eps)=o(\eps)$ if $\lim\limits_{\eps\rightarrow0^+}\abs{f(\eps)\over\eps}=0$. Denote by $\partial K$ the boundary of a convex body $K$. The origin in the Euclidean space $\R^n$ of any dimension is always denoted by $0$. If $A$ is a measurable set of dimension $k$ in $\R^n$, we denote by $|A|$ the $k$-dimensional volume (Lebesgue measure) of $A$. There should be no confusion with the notation for the Euclidean norm of a vector $x\in\R^n$, which is $|x|=\sqrt{\brA{x,x}}$. In the paper we work in the Euclidean space $\R^n$ of a fixed dimension $n$ with the standard basis $\set{e_1,\cdots,e_n}$.\\

\noindent{\bf Acknowledgment}. It is with great pleasure that we thank Fedor Nazarov and Artem Zvavitch for their helpful advice during the preparation of this paper.

\section{General construction on Polytopes}\label{general_construction}

In this section we fix a convex polytope $P$ in $\R^n$ containing the origin in its interior, and we restate the key steps from \cite{NPRZ} for the polytope $P$. 

A subset $F$ of $P$ is called a {\it face} of $P$ if there exists a supporting hyperplane $A$ of $P$ such that $F=P\cap A$. In particular, a face $P$ is called a {\it vertex} if its dimension is $0$, and  a {\it facet} if its dimension is $n-1$. \hide{Notice that every face can have dimension from $0$ to $n-1$.} The empty set $\varnothing$ and the entire polytope $P$ are always considered as {\it improper faces} of $P$ of dimensions $-1$ and $n$, respectively. The {\it dual face} $F^*$ of a face $F$ is defined by $$F^*=\Set{y\in P^\circ : \brA{x,y}=1, \,\forall x\in F}.$$ Note that $F^*$ is a face of $P^\circ$ with $\dim F^*=n-1-\dim  F$ (see \cite[Ch.\ 3]{Gr}).  In case of improper faces, we get $P^*=\varnothing$ and $\varnothing^*=P^\circ$; hence the formula $\dim F^*=n-1-\dim F$ works even for improper faces $F$. 
An ordered set $\F=\set{F^0,\ldots,F^k,\cdots,F^{n-1}}$ consisting of $n$ faces of $P$ is called a \emph{flag} over $P$ if each $F^k$ is of dimension $k$ and $F^0 \subset F^1\subset\cdots\subset F^{n-1}$. 

\medskip
For each face $F$ of $P$ (respectively $P^\circ$), choose an affine subspace $A\sub{F}$ satisfying
\begin{enumerate}
\item[] {\bf (a)}\quad $A\sub{F}$ intersects $P$ (respectively $P^\circ$) at a single point, denoted by $c\sub{F}$, in the interior of $F$.
\item[] {\bf (b)}\quad $\dim  A\sub{F} + \dim  F =n-1$.
\end{enumerate} 
In addition, we assume that the affine subspaces $A\sub{F}$, $A\sub{F^*}$ for a face $F$ and its dual face $F^*$ satisfy
\begin{enumerate}
\setcounter{enumi}{2}
\item[] {\bf (c)}\quad $\brA{x,y}=1$ for every $x\in A\sub{F}$, $y\in A\sub{F^*}$.
\end{enumerate}
\begin{center}
\includegraphics[scale=.5,page=1]{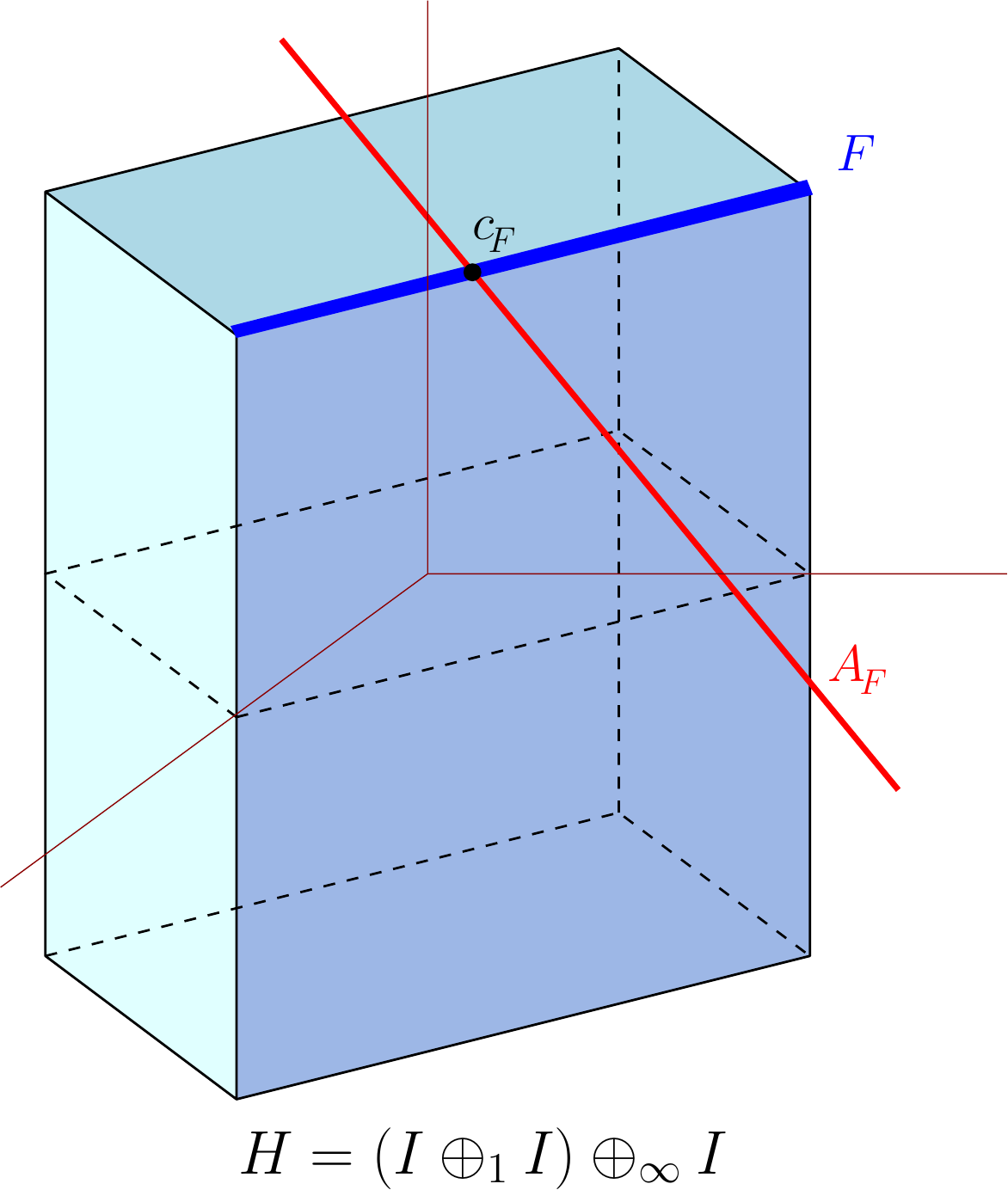}\qquad
\includegraphics[scale=.43,page=2]{fig_construction_3d_polar.pdf}
\end{center}

Let $K$ be a convex body in $\R^n$ containing the origin in its interior. In order to construct the points $x\sub{F}$, $y\sub{F}$ for each face $F$ of $P$, we move the affine subspace $A\sub{F}$ in the direction of $c\sub{F}$ (or equivalently, dilate $A\sub{F}$ from the origin) until it is tangent to $K$. On the affine subspace tangent to $K$ obtained by moving $A\sub{F}$, take the point $x\sub{F}$ on the boundary of $K$ that is nearest to $c\sub{F}$, and the point $y\sub{F}$ on the line containing the origin and $c\sub{F}$. In other words, the points $x\sub{F}$, $y\sub{F}$ are taken so that
\begin{equation}\label{def_points}
x\sub{F} \in (t\sub{F}A\sub{F})\cap\partial K\qquad\text{and}\qquad y\sub{F} = t\sub{F} c\sub{F},
\end{equation}
where $t=t\sub{F}>0$ is chosen so that the affine subspace $tA\sub{F}$ is tangent to $K$. 
For the dual case, we get the points $x\sub{F^*}$, $y\sub{F^*}$ by replacing $P$, $F$, $A\sub{F}$, $K$ with $P^\circ$, $F^*$, $A\sub{F^*}$ $K^\circ$, respectively. 

\begin{definition}\label{def:vol_function}
Let $P$ be a convex polytope in $\R^n$ containing the origin in its interior, $\FF$ the set of all faces of $P$, and $N$ the number of all faces of $P$.
\begin{enumerate}
\item The volume function $V:(\R^n)^N\rightarrow\R_+$ is defined by $$V(Z)=\sum_{\F\text{: flag over $P$}} \abs{Z_\F}\qquad\text{for } Z=\brR{z\sub{F}}\sub{F\in\FF} \in(\R^n)^N,$$
where $Z_\F$ denotes the simplex defined by 
\begin{equation}\label{def_simplex}
Z_\F=\conv\set{0,z\sub{F^0},\ldots,z\sub{F^{n-1}}}\quad\text{for }Z=(z\sub{F})\sub{F\in\FF},\, \F=\set{F^0,\cdots,F^{n-1}}.
\end{equation}
\item Let $K$ be a convex body in $\R^n$ containing the origin. Suppose that the affine subspaces $A\sub{F}$, $A\sub{F^*}$ for each face $F$ satisfy the conditions (a), (b), (c). Then we consider the points in $\R^{nN}$ defined by
\begin{align*}
&X\,\,\,=\brR{x\sub{F}}\sub{F\in\FF},&&Y\,\,\,=\brR{y\sub{F}}\sub{F\in\FF},&&C\,\,\,=\brR{c\sub{F}}\sub{F\in\FF},\\
&X^*=\brR{x\sub{F^*}}\sub{F\in\FF},&&Y^*=\brR{y\sub{F^*}}\sub{F\in\FF},&&C^*=\brR{c\sub{F^*}}\sub{F\in\FF}.
\end{align*}
where $c\sub{F}$ denotes the unique common point of $F$ and $A\sub{F}$, and $x\sub{F}$, $y\sub{F}$, $x\sub{F^*}$, $y\sub{F^*}$ are the points constructed in \eqref{def_points} from $K$, $K^\circ$. 
\end{enumerate}
\end{definition}
The point $C=(c\sub{F})$ depends on $P$ and $A\sub{F}$'s only, but $X=(x\sub{F})$, $Y=(y\sub{F})$ also depend on $K$. With the notation of \eqref{def_simplex}, each flag $\F$ gives the simplices $X_\F$, $Y_\F$, $C_\F$.  Then, the polytopes $P$, $P^\circ$ can be written as
\begin{equation}\label{eq:division_by_simplices}
P=\bigcup_\F C_\F\quad\text{and}\quad P^\circ=\bigcup_\F C^*_{\F^*}.
\end{equation}
Indeed, we can use the induction on $n$ to prove that every polytope $P$ of dimension $n$ containing the origin satisfies $P=\bigcup_\F C_\F$ whenever $c\sub{F}\in{\rm int}(F)$ for each face $F$ of $P$. Assume that it is true for any $(n-1)$-dimensional polytope containing the origin. Let $P$ be a polytope of dimension $n$ containing the origin. Then, every facet $G$ of $P$ is an $(n-1)$-dimensional polytope, and each flag over $G$ (as a polytope of dimension $n-1$) can be viewed as a flag over $P$ containing $G$. Thus, the inductive assumption gives $G=\bigcup_{\F\ni G}\conv\set{c\sub{F}:F\in\F}$ by viewing each facet $G$ as an $(n-1)$-dimensional polytope and the point $c^{}_{G}$ as the origin. The convex hull of the origin and $G$ is the same as that of the origin and $c\sub{F}$'s for any face $F\in\bigcup_{\F\ni G}\F$. Thus, $\conv(\set{0}\cup G)=\bigcup_{\F\ni G}\conv(\set{0}\cup\set{c\sub{F}:F\in\F})$ because the right hand side is already convex. Finally, 
\begin{align*}
P &=\bigcup_G \bigcup_{\F\ni G}\conv(\set{0}\cup\set{c\sub{F}:F\in\F})=\bigcup_G \bigcup\limits_{\F\ni G} C_\F=\bigcup\limits_\F C_\F.
\end{align*}
Moreover, the interiors of any two simplices $C_\F$, $C_{\F'}$ in the above union are disjoint. Indeed, if the facet in a flag $\F$ is different from that of another flag $\F'$, then the intersection of $\conv\set{c\sub{F}:F\in\F}$ and $\conv\set{c\sub{F}:F\in\F'}$ is at most $(n-2)$-dimensional. Thus, the intersection of $C_\F$ and $C_{\F'}$ is at most $(n-1)$-dimensional, which implies that the interiors of  $C_\F$ and $C_{\F'}$ are disjoint. On the other hand, if $\F$ and $\F'$ contain a common facet, then the facet can be viewed as an $(n-1)$-dimensional polytope and the $c\sub{F}$-point as the origin. Using the same (inductive) argument as above, we conclude that the interiors of any different simplices $C_\F$, $C_{\F'}$ are disjoint. Furthermore, if every $x\sub{F}$-point is close enough to the $c\sub{F}$-point, then the above argument can be applied to prove that the interiors of any two simplices $X_\F$, $X_{\F'}$ are also disjoint. In that case, $V(X)$ can be viewed as the volume of the (star-shaped, not necessary convex) polytope, defined by the union of simplices $X_\F$'s over all flags.

\begin{proposition}\label{prop:construction}
Let $P$ be a convex polytope, and $K$ a convex body in $\R^n$ such that $0\in{\rm int}(P\cap K)$. Let $A\sub{F}$, $A\sub{F^*}$ be affine subspaces satisfying the conditions {\rm (a), (b), (c)}, and let $x\sub{F}$, $y\sub{F}$, $c\sub{F}$, $x\sub{F^*}$, $y\sub{F^*}$, $c\sub{F^*}$ be the points obtained  from $K$ and $A\sub{F}$, $A\sub{F^*}$, as in \eqref{def_points}. Then $$\brA{x\sub{F},x\sub{F^*}}=\brA{y\sub{F},y\sub{F^*}}=1.$$ Moreover, if  $\de=\dH(K,P)$ is small enough, then $|x\sub{F}-c\sub{F}|\le c\de$ and $|y\sub{F}-c\sub{F}|\le c\de$ where $c>0$ does not depend on $K$, but may depend on $P$, $F$ and $A\sub{F}$.
\end{proposition}

\begin{proof}
Note first that 
\begin{equation}\label{eq:AF_AFstar}
A\sub{F^*} = \Set{y\in\R^n : \brA{x,y}=1 \,\forall x\in A\sub{F}}.
\end{equation}
Indeed, if we denote the set on the right hand side by $B$, then the condition (c) implies $A\sub{F^*}\subset B$. Moreover, since $\brA{c\sub{F},y}=1$ and $\brA{x-c\sub{F},y}=0$ for each $x\in A\sub{F}$, $y\in B$, we have $B\subset(c\sub{F^*}+c\sub{F}^\perp)\cap(A\sub{F}-c\sub{F})^\perp$. So, the dimension of $B$ is at most $n-1-\dim  A\sub{F}$ which is equal to  $n-1-(n-1-\dim  F)=\dim  F=n-1-\dim  F^*=\dim A\sub{F^*}$ by the condition (b) and the formula $\dim  F^*=n-1-\dim  F$. Thus $A\sub{F^*}=B=\set{y\in\R^n : \brA{x,y}=1 \,\forall x\in A\sub{F}}$.

Let $t=t\sub{F}$ be such that $tA\sub{F}$ is tangent to $K$ at $x\sub{F}$. Consider a hyperplane $A$ containing $tA\sub{F}$ that is tangent to $K$ at $x\sub{F}$. Let $z$ be the dual point of $A$, i.e., $\brA{y,z}=1$ for every $y\in A$. Then $z\in K^\circ$, and $\brA{ty,z}=1$ for all $y\in A\sub{F}$, which implies $tz\in A\sub{F^*}$ by \eqref{eq:AF_AFstar}. Moreover, $\frac1t A\sub{F^*}$ is tangent to $K^\circ$ at $z$. Indeed, every $p\in\frac1t A\sub{F^*}$ satisfies $\brA{tp,q}=1$ for each $q\in A\sub{F}$ and hence $\brA{p,x\sub{F}}=1$, so every $p\in\frac1t A\sub{F^*}$ is not in the interior of $K^\circ$. It implies that $\frac1t A\sub{F^*}$ is tangent to $K^\circ$ at $z$. Therefore, we get $t\sub{F^*}=\frac1t$, which gives $1=t\sub{F}t\sub{F^*}=\brA{y\sub{F},y\sub{F^*}}=\brA{x\sub{F},x\sub{F^*}}$.

For the second part, replacing $\de$ by $c\de$ if necessary, we may assume that $(1-\de)P\subset K\subset(1+\de)P$. Then $x\sub{F}\in tA\sub{F}\cap\partial K$ and $y\sub{F}=tc\sub{F}$ for some $t$ with $1-\de\le t\le1+\de$. So, $|y\sub{F}-c\sub{F}|=|t-1|\cdot|c\sub{F}|\le c_1\de$. Consider $P_1=P\,\cap\,\spn (A\sub{F})$. Then $c\sub{F}$ is a vertex of $P_1$ because $c\sub{F}$ is the unique common point of $P_1$ and $A\sub{F}$. It implies that there exists a constant $c_2$, depending on $P$, $F$, $A\sub{F}$, such that for every $\de>0$, $${\rm diam}\brR{P\cap(1-\de)A\sub{F}}={\rm diam}\brR{P_1\cap(1-\de)A\sub{F}}\le c_2\de.$$ Finally, since $x\sub{F}$, $y\sub{F}$ are contained in $(1+\de)P$ and also lie on $tA\sub{F}$,
\begin{align*}
|x\sub{F}-y\sub{F}|&\le{\rm diam}\Big((1+\de)P\cap tA\sub{F}\Big)\le(1+\de)\cdot {\rm diam}\brR{P\cap \frac{t}{1+\de}A\sub{F}}\\
&\le c_2(1+\de)\brR{1-\frac{t}{1+\de}}\le 2c_2\de,
\end{align*}
which implies $|x\sub{F}-c\sub{F}|\le (c_1+2c_2)\de$.
\end{proof}

\begin{proposition}\label{prop:split_equal_volumes}
Suppose that $P$, $P^\circ$ are divided into simplices $C_\F$, $C^*_\F$ of equal volume, i.e., 
\begin{equation}\label{equal_volumes}
\abs{C_\F}=\abs{C_{\F'}}\quad\text{and}\quad\abs{C^*_\F}=\abs{C^*_{\F'}}\quad\text{for any two flags $\F$, $\F'$}.
\end{equation}
Then 
$V(Y)\cdot V(Y^*) \ge V(C)\cdot V(C^*)=\abs{P}\cdot\abs{P^\circ}$.
\end{proposition}

\begin{proof}
For each flag $\F=\set{F^0,\ldots,F^{n-1}}$ over $P$, 
\begin{align*}
\abs{Y_\F}\abs{Y^*_{\F}} &=\frac1{n!^2}\,\Abs{\det\big(y\sub{F^0},\ldots,y\sub{F^{n-1}}\big)}\cdot \Abs{\det\big(y\sub{(F^0)^*},\ldots,y\sub{(F^{n-1})^*}\big)} \\
&=\frac1{n!^2}\,\Abs{\det\big(t\sub{F^0}c\sub{F^0},\ldots,t\sub{F^{n-1}}c\sub{F^{n-1}}\big)}\cdot \Abs{\det\big(t\sub{(F^0)^*}c\sub{(F^0)^*},\ldots,t\sub{(F^{n-1})^*}c\sub{(F^{n-1})^*}\big)} \\
&=\frac1{n!^2}\Big(\prod_{k=0}^{n-1}t\sub{F^k} t\sub{(F^k)^*}\Big) \Abs{\det\big(c\sub{F^0},\ldots,c\sub{F^{n-1}}\big)}\cdot \Abs{\det\big(c\sub{(F^0)^*},\ldots,c\sub{(F^{n-1})^*}\big)},
\end{align*}
which is equal to $\abs{C_\F}\abs{C^*_{\F}}$ because $t\sub{F}t\sub{F^*}=\brA{y\sub{F},y\sub{F^*}}=1$ for any $F$ by Proposition \ref{prop:construction}.
From Cauchy-Schwarz inequality and the assumption \eqref{equal_volumes}, $$V(Y)V(Y^*)= \brR{\sum_\F \abs{Y_\F}} \brR{\sum_\F\abs{Y^*_{\F}}} \ge \brS{\sum_\F \abs{Y_\F}^{\frac12} \abs{Y^*_{\F}}^{\frac12}}^2 =  \brS{\sum_\F \abs{C_\F}^{\frac12} \abs{C^*_{\F}}^{\frac12}}^2.$$ Since $\abs{C_\F}\abs{C^*_\F}$ is constant for each flag $\F$, we have $$V(Y)V(Y^*)\ge \sum_\F \abs{C_\F} \sum_\F\abs{C^*_\F} = \abs{P}\abs{P^\circ}.$$
\end{proof}

The directional derivative of the function $V$ along a vector $Z=(z\sub{F})$ at a point $A=(a\sub{F})$, which we denote by $\brA{V'(A),Z}$, is defined as $$\brA{V'(A),Z}=\lim_{t\rightarrow0}\frac{V(A+tZ)-V(A)}{t}.$$

\begin{proposition}\label{prop:vol_difference}
Let $P$ be a convex polytope in $\R^n$ containing the origin in its interior. For each face $F$ of $P$, choose an affine subspace $A\sub{F}$ satisfying the conditions {\rm (a), (b), (c)}, and consider the point $C=(c\sub{F})$ given in Definition \ref{def:vol_function}. Suppose that 
\begin{equation}\label{differential_zero}
\brA{V'(C),Z}=0 \quad\text{for every }Z=\big(z\sub{F}\big)\text{ with all } z\sub{F}\in A\sub{F}-A\sub{F}.
\end{equation}
Then $|V(X)-V(Y)|\le c\de^2$, provided that $\de=\dH(K,P)$ is small enough, where $c>0$ is a constant depending on $P$ and $A\sub{F}$'s.
\end{proposition}

\begin{proof}
The Talyor series expansion of $V$ around the point $C=(c\sub{F})$ gives
$$V(X) = V(C) + \brA{V'(C),X-C} + O(|X-C|^2)$$
and
$$V(Y) = V(C) + \brA{V'(C),Y-C} + O(|Y-C|^2)$$
By subtracting the above two equations, the assumption \eqref{differential_zero} gives
\begin{align*}
V(X) - V(Y) &= \brA{V'(C),X-Y} + O(|X-C|^2) + O(|Y-C|^2) \\
&= O(|X-C|^2) + O(|Y-C|^2),
\end{align*}
which is $O(\de^2)$ by Proposition \ref{prop:construction}.
\end{proof}

\begin{remark}[{\it Plan for the proof of Main theorem }\!\!]
Note that the conditions \eqref{equal_volumes}, \eqref{differential_zero} in Propositions \ref{prop:split_equal_volumes}, \ref{prop:vol_difference} depend on the polytope $P$ only. If a polytope $P$ satisfies these two conditions, then Propositions \ref{prop:construction}, \ref{prop:split_equal_volumes}, \ref{prop:vol_difference} give $$V(X)V(X^*)\ge |P||P^\circ|-c\de^2,$$
where $\de=\dH(K,P)$ is small enough and $c>0$ is a constant depending on $P$ and $A\sub{F}$'s only.
In Section \ref{Hanner_polytopes}, we fix a Hanner polytope $H$, and define affine subspaces $A\sub{F}$'s associated with $H$ that satisfy the conditions (a), (b), (c). The special choice of $A\sub{F}$ is not essential to prove the condition \eqref{equal_volumes}, but is necessary for the proof of the condition \eqref{differential_zero} (see Section \ref{differential_volume_function}). In fact, the condition \eqref{equal_volumes} is always true whenever the centroid of each face $F$ of $H$ is chosen as the unique common point $c\sub{F}$ of $A\sub{F}$ and $H$. \hide{Due to condition \eqref{differential_zero}, however,  we need a very particular choice of the affine subspace $A\sub{F}$, which is given in Section \ref{differential_volume_function}.}
To complete the proof, we prove in Section \ref{delta_gap} that there exists a symmetric convex body $\tilde{K}$ with $\dBM(\tilde{K},K)=1+o(\de)$ such that $$|\tilde{K}||\tilde{K}^\circ|\ge V(\tilde{X})V(\tilde{X}^*)+c'\de,$$ where $\tilde{X}$, $\tilde{X}^*$ are the sets of the $x\sub{F}$-points constructed from $\tilde{K}$, $\tilde{K}^\circ$, as in Definition \ref{def:vol_function}.
\end{remark}

\section{Direct sums of Polytopes}\label{direct_sums}

In this section we investigate basic properties of the direct sum (mostly the $\ell_1$ or $\ell_\infty$ sum) of two polytopes to get the general form of a face and of its centroid. In the end of this section we provide a description of the flags over the $\ell_1$ or $\ell_\infty$ sum of two polytopes in terms of the summands.

For $1\le p<\infty$, the \emph{$\ell_p$-sum} $A+_p B$ of two subsets $A$, $B$ of $\R^n$ is defined as the set of all points of the form $x$, $y$, or $\la^{1/p}x+(1-\la)^{1/p}y$ for $x\in A$, $y\in B$ and $\la\in(0,1)$. In particular, if $A$, $B$ are convex, then their $\ell_1$-sum is given by $$A+_1B=\conv(A\cup B).$$ The \emph{$\ell_\infty$-sum} of $A$ and $B$ is defined as the set of all points of the form $x+y$ for $x\in A$, $y\in B$ (the Minkowski sum), i.e., $$A+_\infty B=A+B.$$ We write $A\oplus_p B$ instead of $A+_p B$ if $\spn(A)\cap\spn(B)=\set{0}$. In addition, in case of $p=\infty$, we mostly use the notation $A+B$, $A\oplus B$ without subscript.

\begin{lemma} \label{lem:dim_sum} Let $A$, $B$ be non-empty convex subsets of $\R^{n_1}$, $\R^{n_2}$ respectively, and let $\R^n=\R^{n_1}\oplus\R^{n_2}$. Then $$\dim (A\oplus B)=\dim A + \dim B.$$ If $\aff(A)\cap\aff(B)=\varnothing$ (i.e., at least one of $\aff(A)$ and $\aff(B)$ does not contain the origin), then $$\dim (A\oplus_1 B)=\dim A + \dim B+1.$$ In addition, the second formula holds even if one of the summands is the empty set under the convention $\dim(\varnothing)=-1$. 
\end{lemma}

\begin{proof}
If one of the summands is the empty set, say $B=\varnothing$, then $\dim(A\oplus_1\varnothing)=\dim A=\dim A+\dim\varnothing+1$.
If the dimensions of non-empty sets $A$, $B$ are $k_1$, $k_2$ respectively, then there exist affinely independent sets $\set{a_0,a_1,\ldots,a_{k_1}}\subset A$ and $\set{b_0,b_1,\ldots,b_{k_2}}\subset B$. Consider the sets $$S_1=\Set{a_0,a_1,\ldots,a_{k_1},b_0,b_1,\ldots,b_{k_2}}\subset A\oplus_1 B$$ and $$S_\infty=\Set{a_0+b_0, a_0+b_1,\ldots,a_0+b_{k_2},a_1+b_0,\ldots,a_{k_1}+b_0}\subset A\oplus_\infty B.$$ Then, $\aff(S_1)$ contains $A$, $B$, and also their convex hull, so $\aff(S_1)\supset A\oplus_1B$. In addition, $\aff(S_\infty)$ contains $a_0+\aff(B)$, $\aff(A)+b_0$ and their convex hull, so $\aff(S_\infty)\supset A\oplus_\infty B$. So, it remains to prove that each of the sets $S_1$ and $S_\infty$ is affinely independent. To verify the independence of $S_1$, assume that $0\not\in\aff(A)$ and $$\sum_{j=0}^{k_1}\la_j a_j - \sum_{j=0}^{k_2}\mu_j b_j=0$$ where $\la_j$'s and $\mu_j$'s satisfy $\sum_{j=0}^{k_1}\la_j=\sum_{j=0}^{k_2}\mu_j$. It implies $\sum_{j=0}^{k_1}\la_j a_j = \sum_{j=0}^{k_2}\mu_j b_j=0$ because the origin is the only common point of $\spn(A)$ and $\spn(B)$. Since $\aff(A)$ does not contain the origin, then $\sum_{j=0}^{k_1}\la_j=0$; otherwise, $0=\frac{\sum_{j=0}^{k_1}\la_j a_j}{\sum_{j=0}^{k_1}\la_j}\in\aff(A)$. Thus $\sum_j\la_j=0$ and $\sum_j\mu_j=0$. The affine independence of each of $\set{a_0,\ldots,a_{k_1}}$ and $\set{b_0,\ldots,b_{k_2}}$ gives $\la_0=\cdots=\la_{k_1}=0$ and $\mu_0=\cdots=\mu_{k_2}=0$. 
We now prove the affine independence of $S_\infty$. Assume that 
\begin{equation}\label{eq:ell_infty_affine_indep}
\al(a_0+b_0)+\sum_{j=1}^{k_1}\la_j(a_j+b_0) + \sum_{j=1}^{k_2}\mu_j(a_0+b_j)=0
\end{equation}
where $\al$, $\la_j$'s and $\mu$'s satisfy $\al+\sum_{j=1}^{k_1}\la_j+\sum_{j=1}^{k_2}\mu_j=0$. Projecting the equality \eqref{eq:ell_infty_affine_indep} onto $\R^{n_1}$ and $\R^{n_2}$ respectively, we get $$\brR{\al+\sum_{j=1}^{k_2}\mu_j} a_0 + \sum_{j=1}^{k_1}\la_j a_j=0\quad\text{and}\quad\brR{\al+\sum_{j=1}^{k_1}\la_j} b_0 + \sum_{j=1}^{k_2}\mu_j b_j=0.$$ Finally, the affine independence of each of $\set{a_0,\ldots,a_{k_1}}$ and $\set{b_0,\ldots,b_{k_2}}$ gives $\la_1=\cdots=\la_{k_1}=0$, $\mu_1=\cdots=\mu_{k_2}=0$, and $\al=0$. 
\end{proof}

\begin{lemma} \label{centroid_formula}
Let $A$, $B$ be non-empty convex subsets in $\R^{n_1}$, $\R^{n_2}$ respectively, and let $\R^n=\R^{n_1}\oplus\R^{n_2}$. Then the centroid of the $\ell_\infty$-sum of $A$, $B$ is given by $$c(A\oplus B)= c(A) + c(B),$$ where $c(A)$ denotes the centroid of $A$. If $\aff(A)\cap\aff(B)=\varnothing$ (i.e., at least one of $\aff(A)$ and $\aff(B)$ does not contain the origin), then the centroid of the $\ell_1$-sum of $A$, $B$ is $$c(A\oplus_1 B)=\frac{\dim A +1}{\dim A+\dim B +2}\,c(A) + \frac{\dim B +1}{\dim A+\dim B +2}\,c(B).$$ In addition, the second formula holds even if one of the summands is the empty set $\varnothing$ under the conventions $\dim (\varnothing)=-1$ and $c(\varnothing)=0$.
\end{lemma}

\begin{proof}
Since $c(TA)=Tc(A)$ for any invertible linear transformation $T$ on $\R^n$, we may assume that $\R^n=\R^{n_1}\oplus\R^{n_2}$ is the orthogonal sum of $\R^{n_1}$ and $\R^{n_2}$. Let $k_1=\dim A$, $k_2=\dim B$. First, consider the $\ell_\infty$-sum case. If, say, $k_1=0$, then $c(A\oplus B)=c(c(A)+B)=c(A)+c(B)$. So, we may assume $k_1,k_2\ge1$. The convex sets $A$, $B$ can be viewed as subsets of $\R^{k_1}$, $\R^{k_2}$ respectively. Then
\begin{align*}
c(A\oplus B) &= \frac1{|A\oplus B|}\int_A\int_B (x+y)\,dxdy = \frac1{|A||B|}\int_A\int_B (x+y)\,dxdy \\
&= \frac1{|A|}\int_A x\,dx + \frac1{|B|}\int_B y\,dy = c(A)+c(B),
\end{align*}
where $dx$, $dy$ are the Lebesgue measures on $\R^{k_1}$, $\R^{k_2}$.

For the $\ell_1$-sum case, let $E=\aff(A)+\aff(B)-c(B)$, and $\theta=c(B)-c(A)$. Then, $\te\not\in E-E$ because $E-E-\te=\aff(A)+\aff(B)-c(A)-c(B)-(c(B)-c(A))=\aff(A)-\aff(B)$ does not contain the origin, and 
\begin{align*}
t\theta+E &=t(c(B)-c(A))+\aff(A)+\aff(B)-c(B)\\
&=(1-t)\,\aff(A)+t\,\aff(B).
\end{align*} 
Choose $\varphi\in\Sp\cap\spn(E\cup\set{\te})$ which is orthogonal to $E$.
Thus, from $A\oplus_1B\subset\bigcup\limits_{0\le t\le1}(t\te+E)$, we get $$|A\oplus_1B|=\int_0^1\Abs{(A\oplus_1B)\cap(t\theta+E)}\,\abs{\brA{\theta,\varphi}}\,dt.$$
Since 
\begin{align*}
(A\oplus_1B)\cap(t\theta+E)&=\bigcup_{0\le s\le1}((1-s)A+sB)\cap((1-t)\aff(A)+t\aff(B))\\
&=(1-t)A+tB,
\end{align*}
we get
\begin{align*}
|A\oplus_1B|&=\int_0^1\Abs{(1-t)A+tB}\,\abs{\brA{\theta,\varphi}}\,dt = \abs{\brA{\theta,\varphi}}\int_0^1 \Abs{(1-t)A}\Abs{tB}\,dt\\
&=\abs{\brA{\theta,\varphi}}|A||B|\int_0^1 (1-t)^{k_1}t^{k_2}\,dt = \frac{k_1!k_2!\abs{\brA{\theta,\varphi}}|A||B|}{(k_1+k_2+1)!}.
\end{align*}
Similarly, $$\int_{A\oplus_1B}z\,dz =\int_0^1\brR{\int_{(1-t)A+tB}z\,dz}\abs{\brA{\theta,\varphi}}dt.$$
Note that $\abs{(1-t)A+tB}^{-1}\int_{(1-t)A+tB}x\,dx$ is the centroid of $(1-t)A+tB$, which is equal to $(1-t)c(A)+tc(B)$ by the $\ell_\infty$-case. Thus
\begin{align*}
\int_{A\oplus_1B}x\,dx &=\abs{\brA{\theta,\varphi}}\int_0^1 \Big((1-t)c(A)+tc(B)\Big)\Abs{(1-t)A+tB}\,dt\\
&=\abs{\brA{\theta,\varphi}}|A||B|\brR{c(A)\int_0^1 (1-t)^{k_1+1}t^{k_2}\,dt + c(B)\int_0^1 (1-t)^{k_1}t^{k_2+1}\,dt}\\
& = \frac{k_1!k_2!\abs{\brA{\theta,\varphi}}|A||B|}{(k_1+k_2+1)!}\brR{\frac{k_1+1}{k_1+k_2+2}c(A)+\frac{k_2+1}{k_1+k_2+2}c(B)}
\end{align*}
Dividing the above by $|A\oplus_1B|$, we get the formula for $c(A\oplus_1B)$. The case that a summand contains the empty set is clear.
\end{proof}

The description of extreme points of the $\ell_1$ or $\ell_\infty$ sum of two convex sets is well known (e.g., \cite{We}). By $\ext(A)$ we denote the set of extreme points of a convex set $A$. Then the extreme points of the $\ell_1$ or $\ell_\infty$ sum of two convex sets are described in the following way: if $A\subset\R^{n_1}$, $B\subset\R^{n_2}$ are convex sets and $\R^n=\R^{n_1}\oplus\R^{n_2}$, then 
\begin{equation}\label{eq:extremepoints}
\ext(A\oplus_1B)=\ext(A)\cup\ext(B)\quad\text{and}\quad\ext(A\oplus B)=\ext(A)+\ext(B).
\end{equation}
In particular, if $A$, $B$ are polytopes, then it gives a characterization of the vertices (zero-dimensional faces) of the sum $A\oplus_1B$ or $A\oplus B$. Moreover, it can be generalized to the faces of any dimension as follows.

\begin{lemma} \label{lem:sum_faces}
Let $P_1\subset\R^{n_1}$, $P_2\subset\R^{n_2}$ be convex polytopes containing the origin, and let $\R^n=\R^{n_1}\oplus\R^{n_2}$. Then, $F$ is a face of $P_1\oplus_1P_2$ if and only if $F$ is of the form $$F_1\oplus_1\varnothing,\quad\varnothing\oplus_1F_2,\quad\text{or}\quad F_1\oplus_1 F_2,$$
where $F_1$, $F_2$ are faces of $P_1$, $P_2$. In case of the $\ell_\infty$-sum, $F$ is a face of $P_1\oplus_\infty P_2$ if and only if $F$ is of the form $$F_1\oplus P_2,\quad P_1\oplus F_2,\quad\text{or}\quad F_1\oplus F_2,$$
where $F_1$, $F_2$ are faces of $P_1$, $P_2$. 
\end{lemma}
\begin{proof}
First, consider the $\ell_1$-sum case. Let $F$ be a face of $P_1\oplus_1P_2$. Then there exists a supporting hyperplane $A$ of $P_1\oplus_1P_2$ such that $F=A\cap(P_1\oplus_1P_2)$. By \eqref{eq:extremepoints}, every face $F$ of $P_1\oplus_1P_2$ can be expressed as $F=F_1\oplus_1F_2$ where $F_j$ is the convex hull of some of extreme points of $P_j$ for $j=1,2$. In addition, for each $j=1,2$,
\begin{align*}
(A\cap\R^{n_j})\cap P_j &= A\cap(\R^{n_j}\cap(P_1\oplus_1P_2))=\R^{n_j}\cap (A\cap(P_1\oplus_1P_2))=\R^{n_j}\cap F\\
&=\R^{n_j}\cap (F_1\oplus_1F_2) = F_j.
\end{align*}
If $F_j\neq\varnothing$, then $A\cap\R^{n_j}$ is a hyperplane in $\R^{n_1}$ with $(A\cap\R^{n_j})\cap P_j=F_j$, which means that $F_j$ is a face of $P_j$. 

For the $\ell_\infty$-sum case, let $F$ be a face of $P_1\oplus_\infty P_2$. Then there exists a hyperplane $A$ of $P_1\oplus_\infty P_2$ with $F=A\cap(P_1\oplus_\infty P_2)$. Write $A=x_1+x_2+\te^\perp$ for $x_1\in P_1$, $x_2\in P_2$ with $x_1+x_2\in F$. Note that $\brA{y_1+x_2,\te}\le\brA{x_1+x_2,\te}$ and $\brA{x_1+y_2,\te}\le\brA{x_1+x_2,\te}$ for each $y_1\in P_1$, $y_2\in P_2$. It implies $\brA{y_1,\te}\le\brA{x_1,\te}$ and $\brA{y_2,\te}\le\brA{x_2,\te}$ for each $y_1\in P_1$, $y_2\in P_2$. Thus, $\brA{y_1+y_2,\te}=\brA{x_1+x_2,\te}$ if and only if $\brA{y_1,\te}=\brA{x_1,\te}$ and $\brA{y_1,\te}=\brA{x_1,\te}$. In other words,
\begin{align*}
F &= (P_1\oplus_\infty P_2)\cap A = (P_1+P_2)\cap(x_1+x_2+\te^\perp)\\
&= P_1\cap(x_1+\te^\perp) + P_2\cap(x_2+\te^\perp).
\end{align*}
Let $F_j=P_j\cap(x_j+\te^\perp)$ for each $j=1,2$. If $\te\in\R^{n_j}$ for $j=1$ or $2$, then $F_j$ is a face of $P_j$ and the other summand is an improper face. Otherwise, each $F_j$ is a face of $P_j$ for $j=1,2$. 
\end{proof}

\begin{lemma} \label{lem:sum_dual_face} Let $P_1\subset\R^{n_1}$, $P_2\subset\R^{n_2}$ be convex polytopes containing the origin in their interiors, and let $\R^n=\R^{n_1}\oplus\R^{n_2}$. Then, the dual face of a face of the $\ell_1$ or $\ell_\infty$ sum of $P_1$ and $P_2$ is given by
\begin{align*}
&(F_1\oplus_1 \varnothing)^*=F_1^* \oplus P_2^\circ,&&(\varnothing\oplus_1 F_2)^*=P_1^\circ \oplus F_2^*,&&(F_1\oplus_1 F_2)^*=F_1^* \oplus F_2^*,\\
&(F_1\oplus P_2)^*=F_1^* \oplus_1 \varnothing,&&(P_1\oplus F_2)^*=\varnothing\oplus_1 F_2^*,&&(F_1\oplus F_2)^*=F_1^* \oplus_1 F_2^*, 
\end{align*}
where $F_1$, $F_2$ are faces of polytopes $P_1$, $P_2$ respectively, and $F^*$ denotes the dual face of a face $F$.
\end{lemma}
\begin{proof}
Let $F=F_1\oplus_1F_2$ where each $F_j$ is $\varnothing$ or a face of $P_j$ for $j=1,2$. Then, by Lemma \ref{lem:sum_faces}, $F_1^*\oplus F_2^*$ is a face of $P_1^\circ\oplus P_2^\circ$, and clearly $F_1^*\oplus F_2^*\subset F^*$. Moreover, it follows from Lemma \ref{lem:dim_sum} and the formula $\dim F^*=n-1-\dim F$ that the dimension of $F_1^* \oplus F_2^*$  is the same as that of $F^*$. So, $\aff(F_1^*\oplus F_2^*)=\aff(F^*)$. It implies
\begin{align*}
(F_1\oplus_1F_2)^* &= F^*=\aff(F^*)\cap(P_1\oplus_1P_2)^\circ =\aff(F_1^*\oplus F_2^*)\cap(P_1^\circ\oplus P_2^\circ) \\
&=(\aff(F_1^*)\cap P_1^\circ)\oplus(\aff(F_2^*)\cap P_2^\circ) =F_1^*\oplus F_2^*.
\end{align*}
whenever each $F_j$ is $\varnothing$ or a face $P_j$ for $j=1,2$. In addition, replacing $F_1^*$, $F_2^*$ with $G_1$, $G_2$, we get $(G_1^*\oplus G_2^*)^*=G_1\oplus_1 G_2$. It implies $(G_1\oplus_1 G_2)^*=G_1^*\oplus G_2^*$ whenever each $G_j$ is $P_j$ or a face $P_j$.
\end{proof}

\begin{remark}[{{\it More on the faces}}]\label{rmk:face_form} 
We describe the faces of the sum of two polytopes. Lemma \ref{lem:dim_sum} and Lemma \ref{lem:sum_faces} make it possible to describe the faces over the sum with the dimensions of summands. Let $P_1\subset\R^{n_1}$, $P_2\subset\R^{n_2}$ be convex polytopes, and let $\R^n=\R^{n_1}\oplus\R^{n_2}$. 
For each $k=1,\ldots,n$, every face $F^{k-1}$ of dimension $k-1$ of $P_1\oplus_1 P_2$ is of the form
\begin{equation}\label{face_form1}
F^{k-1}=F_1^{k_1-1}\oplus_1 F_2^{k_2-1},
\end{equation}
and every face $F^{n-k}$ of dimension $n-k$ of $P_1\oplus_\infty P_2$ is of the form
\begin{equation}\label{face_form2}
F^{n-k}=F_1^{n_1-k_1}\oplus F_2^{n_2-k_2},
\end{equation}
where $0\le k_1\le n_1$, $0\le k_2\le n_2$, $k_1+k_2=k$, and each $F_j^d$ denotes a face or an improper face of $P_j$ of dimension $d$ for $j=1,2$.
\end{remark}

If a face $F=F_1\oplus_1 F_2$ (respectively $F_1\oplus F_2$) is contained in another face $G=G_1\oplus_1 G_2$ (respectively $G_1\oplus G_2$), then it follows from \eqref{eq:extremepoints} that $F_1\subset G_1$ and $F_2\subset G_2$. Furthermore, if $F\subset G$ with $\dim G=\dim F+1$, then we have, by Lemma \ref{lem:dim_sum}, either $$F_1=G_1,\, F_2\subsetneq G_2,\, \dim G_2=\dim F_2+1,\quad\text{or}\quad F_2=G_2,\, F_1\subsetneq G_1,\, \dim G_1=\dim F_1+1.$$ 
We introduce the following notion to describe a flag over the sum of two polytopes.

\begin{definition}\label{def:type}
Let $P_1$, $P_2$ be convex polytopes in $\R^{n_1}$, $\R^{n_2}$ respectively, and $\R^n=\R^{n_1}\oplus\R^{n_2}$.
A flag $\F=\set{F^0,\ldots,F^{n-1}}$ over $P_1\oplus_1P_2$ is of {\it type $(j_1,\ldots,j_n)\in\set{1,2}^n$} if each element of $\F$ is obtained in the following way:
\begin{enumerate}
\item The face $F^0$ of dimension $0$ is the $\ell_1$-sum of a zero-dimensional face of $P_{j_1}$ and $\varnothing$.
\item For $2\le k\le n$, the face $F^{k-1}$ of dimension $k-1$ is obtained from the $\ell_1$-decomposition of $F^{k-2}$ by increasing the dimension of the $j_k$-th summand (a face of $P_{j_k}$), i.e., if $F^{k-2}=F_1\oplus_1F_2$, then 
\begin{equation*}
F^{k-1}=\begin{cases}
\widetilde{F}_1\oplus_1F_2,&\quad\text{if }j_k=1\\
F_1\oplus_1\widetilde{F}_2,&\quad\text{if }j_k=2
\end{cases}
\end{equation*}
where each $\widetilde{F}_j$ is a face of $P_j$ of dimension $1+\dim F_j$ containing $F_j$ for $j=1,2$.
\end{enumerate}
In case of the $\ell_\infty$ sum, we say that a flag $\F=\set{F^0,\ldots,F^{n-1}}$ over $P_1\oplus_\infty P_2$ is of \emph{type $(j_1,\ldots,j_n)$} if its dual flag $\F^*=\set{(F^{n-1})^*,\ldots,(F^0)^*}$ is of type $(j_1,\ldots,j_n)$, as a flag over $P_1^\circ\oplus_1P_2^\circ$.
\end{definition}

\begin{remark}[{{\it Description of flags}}]\label{rmk:flag_form} 
Note that every flag over the $\ell_1$ or $\ell_\infty$ sum in $\R^n=\R^{n_1}\oplus\R^{n_2}$ of $P_1\subset\R^{n_1}$ and $P_2\subset\R^{n_2}$ has a type as an $n$-tuple consisting of $n_1$ copies of $1$ and $n_2$ copies of $2$, that is, an element of
\begin{equation}\label{flag_tuples}
J=\set{(j_1,\ldots,j_n)\in\set{1,2}^n\Big|\,\,\substack{\#\set{k:j_k=1}=n_1\\ \#\set{k:j_k=2}=n_2}\,}.
\end{equation}
In addition, for $p=1$ or $\infty$, every flag $\F$ over the $\ell_p$-sum of $P_1$ and $P_2$ gives two flags $\F_1$, $\F_2$ in lower dimensions, defined by
\begin{align}\label{lower_flags}
\begin{split}
\F_1&=\set{F_1 \,\Big|\, F_1\oplus_p F_2\in\F\text{ for some $F_2$},\,\,0\le\dim F_1<n_1}\\
\F_2&=\set{F_2 \,\Big|\, F_1\oplus_p F_2\in\F\text{ for some $F_1$},\,\,0\le\dim F_2<n_2} .
\end{split}
\end{align}
Every flag $\F$ is uniquely determined by lower dimensional flags $\F_1$, $\F_2$ and a type $\sigma\in J$. Moreover, two flags $\F_1=\{F_1^0,\ldots,F_1^{n_1-1}\}$, $\F_2=\{F_2^0,\ldots,F_2^{n_2-1}\}$ in lower dimensions and a type $\sigma=(j_1,\ldots,j_n)\in J$ can be associated with the flag $\F=\{F^0,\ldots,F^{n-1}\}$ defined by, for $1\le k\le n$,
\begin{align}
F^{k-1}&=F_1^{\sigma_1(k)-1}\oplus_1 F_2^{\sigma_2(k)-1} \tag{$\ell_1$ -sum case} \\
F^{n-k}&=F_1^{n_1-\sigma_1(k)}\oplus F_2^{n_2-\sigma_2(k)} \tag{$\ell_\infty$-sum case}
\end{align}
where $\sigma_j(k)$ denotes the number of $j$'s among the first $k$ entries $j_1,\ldots,j_k$ of $\sigma=(j_1,\ldots,j_n)$ for each $j=1,2$. We can also see that this flag constructed from $\F_1$, $\F_2$ and $\sigma$ is of type $\sigma$ and has the lower dimensional flags $\F_1$, $\F_2$. Consequently, every flag $\F$ over the $\ell_1$ or $\ell_\infty$ sum of $P_1$ and $P_2$ can be considered as an triple of a type $\sigma\in J$ and two lower dimensional flags $\F_1$, $\F_2$ over $P_1$, $P_2$.
\end{remark}

\begin{example}
Let $\F_1=\set{F_1^0,F_1^1,F_1^2,F_1^3}$ and $\F_2=\set{F_2^0,F_2^1,F_2^2}$ be flags over a polytope $P_1\subset\R^4$ and a polytope $P_2\subset\R^3$ respectively. Then $$\F=\Set{ F_1^0\oplus_1\varnothing, F_1^1\oplus_1\varnothing, F_1^1\oplus_1 F_2^0,  F_1^2\oplus_1 F_2^0,  F_1^2\oplus_1 F_2^1,  F_1^2\oplus_1 F_2^2,  F_1^3\oplus_1 F_2^2 }$$ is a flag over the $\ell_1$-sum of type $(1,1,2,1,2,2,1)$. For the $\ell_\infty$-sum case, consider the flag $$\F'=\Set{F_1^0\oplus F_2^0,  F_1^1\oplus F_2^0,  F_1^1\oplus F_2^1,  F_1^1\oplus F_2^2, F_1^2\oplus F_2^2, F_1^2\oplus P_2, F_1^3\oplus P_2}.$$ By taking the dual flag, we can see that it has the same type as $\F'$, which is $(1,1,2,1,2,2,1)$.
\end{example}

\section{Construction on Hanner polytopes}\label{Hanner_polytopes}

A symmetric convex body $H$ is called a {\it Hanner polytope} if $H$ is one-dimensional, or it is the $\ell_1$ or $\ell_\infty$ sum of two (lower dimensional) Hanner polytopes. 
Thus every Hanner polytope in $\R^n$ can be obtained from $n$ symmetric intervals (one-dimensional Hanner polytopes) by taking the $\ell_1$ or $\ell_\infty$ sums $(n-1)$ times. For example, $H=((I_1\oplus_1 I_2)\oplus_\infty I_3)\oplus_1(I_4\oplus_\infty I_5)$ is a Hanner polytope in $\R^5$ obtained from symmetric intervals $I_1,\ldots,I_5\subset\R^5$. Note that every summand in the representation of $H$ in terms of symmetric intervals is also a (lower dimensional) Hanner polytope. 

For each Hanner polytope, we now define very particular affine subspaces $A\sub{F}$ satisfying (a), (b), (c) in Section \ref{general_construction} to guarantee the conditions \eqref{equal_volumes}, \eqref{differential_zero} in Propositions \ref{prop:split_equal_volumes}, \ref{prop:vol_difference}.  
When a representation of a Hanner polytope $H$ in terms of symmetric intervals is given, we define the affine subspaces $A\sub{F}$'s inductively for all lower dimensional Hanner polytopes of $H$ which appear as a summand in the representation of $H$.

\begin{definition}\label{def:AF}
Define the affine subspaces $A\sub{F}$'s inductively for any Hanner polytope $H$ in the following way: 
\begin{enumerate}
\item If $H$ is one-dimensional, $H$ has the only two faces (of dimension zero). For each face $F$ of a one-dimensional Hanner polytope $H$, define the affine subspace $A\sub{F}$ by $A\sub{F}=F.$
\item Let $H$ be a Hanner polytope obtained from two lower dimensional Hanner polytopes $H_1$ and $H_2$ by taking the $\ell_1$ or $\ell_\infty$-sum. Then, for the $\ell_1$ sum case, the affine subspace $A\sub{F}$ for each face of the form $F_1\oplus_1\varnothing$, $\varnothing\oplus_1F_2$, or $F_1\oplus_1F_2$ is defined by 
\begin{align}\label{def_AF1}
&A\sub{F_1\oplus_1\,\varnothing\,} = A\sub{F_1}+\spn(H_2), \qquad A\sub{\,\varnothing\,\oplus_1F_2} = \spn(H_1)+A\sub{F_2},\notag \\
&A\sub{F_1\oplus_1F_2} = \frac{\dim F_1+1}{\dim F_1+\dim F_2+2}\,A\sub{F_1} + \frac{\dim F_2+1}{\dim F_1+\dim F_2+2}\,A\sub{F_2}.
\end{align}
For the $\ell_\infty$ sum case, it is defined by
\begin{align}\label{def_AF2}
&A\sub{F_1\oplus H_2} = A\sub{F_1},\qquad A\sub{H_1\oplus F_2} = A\sub{F_2} \notag \\
&A\sub{F_1\oplus F_2} = A\sub{F_1} + A\sub{F_2} + \spn\Set{\frac{c\sub{F_1}}{\dim H_1\!-\!\dim F_1}-\frac{c\sub{F_2}}{\dim H_2\!-\!\dim F_2}},
\end{align}
where each $c\sub{F_j}$ is the centroid of a face $F_j$ of $H_j$ for $j=1,2$.
\end{enumerate}
\end{definition}

\begin{lemma}
Let $H$ be any Hanner polytope in $\R^n$, and $F$ a face of $H$. Consider the affine subspace $A\sub{F}$ defined as above. Then $A\sub{F}$ satisfies the conditions {\rm(a), (b), (c)} from Section \ref{general_construction}. Moreover, the unique common point $c\sub{F}$ of $H$ and  $A\sub{F}$ is the centroid of $F$.
\end{lemma}
\begin{proof}
We use the induction on the dimension $n$ of $H$. The statement is trivial when $\dim H=1$.
Let $H\subset\R^n$ be the $\ell_1$ or $\ell_\infty$ sum of lower dimensional Hanner polytopes $H_1\subset\R^{n_1}$, $H_2\subset\R^{n_2}$. Assume that the affine subspaces $A\sub{F_j}$ for any face $F_j$ of $H_j$ satisfy the conditions (a), (b), (c) in Section \ref{general_construction}, and $H\cap A\sub{F_j}=\{c\sub{F_j}\}$ for $j=1,2$. First consider the case that a face $F$ of $H$ contains an improper face as a summand. If, say, $F=F_1\oplus_1\varnothing$ or $F_1\oplus H_2$, then the affine subspace $A\sub{F}$ is given by $A\sub{F_1\oplus_1\varnothing}=A\sub{F_1}+\R^{n_2}$ or $A\sub{F_1\oplus H_2}=A\sub{F_1}$. We can see that all conditions (a), (b), (c) hold under the induction hypothesis. For (a), note that $H\cap A\sub{F_1\oplus_1\varnothing}=(H_1\oplus_1H_2)\cap(A\sub{F_1}+\R^{n_2})=H_1\cap A\sub{F_1}$ and $H\cap A\sub{F_1\oplus H_2}=(H_1\oplus_\infty H_2)\cap A\sub{F_1}=H_1\cap A\sub{F_1}$, both of which have the unique element $c\sub{F_1}$. For (b), note that $\dim A\sub{F_1\oplus_1\varnothing}=\dim A\sub{F_1}+n_2=(n_1-1-\dim F_1)+n_2=n-1-\dim F_1\oplus_1\varnothing$ and $\dim A\sub{F_1\oplus H_2}=\dim A\sub{F_1}=n_1-1-\dim F_1=n-1-\dim (F_1\oplus H_2)$. For (c), we have $\brA{x_1+z_2,y_1}=\brA{x_1,y_1}=1$ for every $x_1+z_2\in A\sub{F_1\oplus_1\varnothing}=A\sub{F_1}+\R^{n_2}$ and $y_1\in A\sub{F_1^*\oplus H_2^\circ}=A\sub{F_1^*}$.
Now, it remains to consider the cases: $F=F_1\oplus_1F_2$ or $F_1\oplus F_2$ where each $F_j$ is a face of dimension $k_j$ of $H_j$.

For the condition (a), we can see from Lemma \ref{centroid_formula} that the centroid of any face $F$ belongs to $A\sub{F}$. So, we need to show the uniqueness of the common point of $A\sub{F}$ and $H$. For the $\ell_1$-sum case, let $\frac{k_1+1}{k_1+k_2+2}\,p + \frac{k_2+1}{k_1+k_2+2}\,q$ for $p\in A\sub{F_1}$, $q\in A\sub{F_2}$ be a common point of $H$ and $A\sub{F}$. Then, $$\frac{1}{\la}\cdot\frac{k_1+1}{k_1+k_2+2}\,\,p\in H_1\quad\text{and}\quad\frac{1}{1-\la}\cdot\frac{k_2+1}{k_1+k_2+2}\,\,q\in H_2$$ for some $\la\in(0,1)$ because every point in $H$ is a convex combination of two points in $H_1$, $H_2$. Then $\frac{1}{\la}\frac{k_1+1}{k_1+k_2+2}=1=\frac{1}{1-\la}\frac{k_2+1}{k_1+k_2+2}$. Otherwise, either $\frac{1}{\la}\frac{k_1+1}{k_1+k_2+2}>1$ or $\frac{1}{1-\la}\frac{k_2+1}{k_1+k_2+2}>1$. If $r:=\frac{1}{\la}\frac{k_1+1}{k_1+k_2+2}>1$, then $rA\sub{F_1}\cap H_1$ must be empty because $A\sub{F_1}$ intersects $H$ at a single point by the induction hypothesis. However, it is impossible because $rp\in rA\sub{F_1}\cap H_1$. Thus, $r=1$ gives $p\in H_1\cap A\sub{F_1}$ and $q\in H_1\cap A\sub{F_2}$. By the induction hypothesis, we get $p=c\sub{F_1}$ and $q=c\sub{F_2}$, which give the uniqueness of the common point of $A\sub{F}$ and $H$. 
Now, for the $\ell_\infty$-sum case, suppose that $p+q+t\big(\frac1{n_1-k_1}c\sub{F_1}-\frac1{n_2-k_2}c\sub{F_2}\big)$ for $p\in A\sub{F_1}$, $q\in A\sub{F_2}$ is a common point of $H\cap A\sub{F}$. Then $p+\frac{t}{n_1-k_1}c\sub{F_1}\in H_1$ and $q-\frac{t}{n_2-k_2}c\sub{F_2}\in H_2$ because $H=H_1\oplus_\infty H_2$. It implies that $\big(1+\frac{t}{n_1-k_1}\big)A\sub{F_1}$ intersects $H_1$, and $\big(1-\frac{t}{n_2-k_2}\big)A\sub{F_2}$ intersects $H_2$. But one of them is impossible unless $t=0$ by the same reason as in the $\ell_1$-sum case. So, $t=0$ gives $p\in H_1\cap A\sub{F_1}$, $q\in H_2\cap A\sub{F_2}$, i.e., $p=c\sub{F_1}$, $q=c\sub{F_2}$.

To verify the condition (b), we use \eqref{def_AF1}, \eqref{def_AF2}, and Lemma \ref{lem:dim_sum}. Then
\begin{align*}
\dim A\sub{F} &=
\begin{cases}
\dim A\sub{F_1} + \dim A\sub{F_2}=(n_1\!-\!1\!-\!k_1) + (n_2\!-\!1\!-\!k_2), &\text{if }F=F_1\oplus_1 F_2 \\
\dim A\sub{F_1} + \dim A\sub{F_2} +1=(n_1\!-\!1\!-\!k_1) + (n_2\!-\!1\!-\!k_2) +1, &\text{if }F=F_1\oplus F_2
\end{cases}\\
&=n-1-\dim F
\end{align*}
To verify the condition (c), it suffices to show $\brA{x,y}=1$ for $x\in A\sub{F_1\oplus_1F_2}$, $y\in A\sub{F_1^*\oplus F_2^*}$. Write 
\begin{align*}
x&=\frac{k_1+1}{k_1+k_2+2}\,p + \frac{k_2+1}{k_1+k_2+2}\,q\,\,\,\in A\sub{F_1\oplus_1F_2},\\
y&=p^*+q^*+\frac1{k_1+1}\,c\sub{F_1^*}-\frac1{k_2+1}\,c\sub{F_2^*}\,\,\in A\sub{F_1^*\oplus F_2^*},
\end{align*}
where $p\in A\sub{F_1}$, $q\in A\sub{F_2}$ and $p^*\in A\sub{F_1^*}$, $q^*\in A\sub{F_2^*}$. Thus,
\begin{align*}
\brA{x,y} &=\brA{ \frac{k_1+1}{k_1+k_2+2}p + \frac{k_2+1}{k_1+k_2+2}q, p^*+q^*+\frac1{k_1+1}c\sub{F_1^*}-\frac1{k_2+1}c\sub{F_2^*}}\\
&=\frac{k_1+1}{k_1+k_2+2}\brA{p,p^*}+\frac{k_2+1}{k_1+k_2+2}\brA{q,q^*} + \frac1{k_1+k_2+2}\Big(\langle p,c\sub{F_1^*}\rangle - \langle q,c\sub{F_2^*}\rangle\Big) \\
&=\frac{k_1+1}{k_1+k_2+2}+\frac{k_2+1}{k_1+k_2+2}+\frac1{k_1+k_2+2}(1-1)=1
\end{align*}
because  $\brA{p,p^*}=\brA{q,q^*}=1$ by the induction hypothesis.
\end{proof}

To compute the volume of the simplex $C_\F$ or $X_\F$ defined in \eqref{def_simplex} from $C=(c\sub{F})$, $X=(x\sub{F})$, and a flag $\F$, we need the following lemma.
\begin{lemma}\label{lem:gauss_elimination}
Let $\sigma=(j_1,\ldots,j_n)$ be an $n$-tuple consisting of $n_1$ copies of $1$ and $n_2$ copies of $2$ for $n=n_1+n_2$. For $j=1,2$, denote by $\sigma_j(k)$ the number of $j$'s among the first $k$ entries $j_1,\ldots,j_k$ of $\sigma$. Consider the $n\times n$ matrix $M$ whose rows consist of $$p_{\sigma_1(k)}+q_{\sigma_2(k)}+\xi_kz \quad\text{for }k=1,\ldots,n$$ where $\xi_1,\ldots,\xi_n\in\R$, $p_0=q_0=0,p_1,\ldots,p_{n_1},q_1,\ldots,q_{n_2},z\in\R^n$. Then the absolute value of the determinant of $M$ is the same as that of the matrix $M'$ whose rows consist of
\begin{align*}
p_{k_1}+\phi_1^\sigma(k_1)z &\quad\text{for }k_1=1,\ldots,n_1, \\
q_{k_2}+\phi_2^\sigma(k_2)z &\quad\text{for }k_2=1,\ldots,n_2,
\end{align*}
where $\phi_1^\sigma$, $\phi_2^\sigma$ are functions determined by $\xi_1,\ldots,\xi_n\in\R$ and $\sigma$.
More precisely, for $j=1,2$, the function $\phi_j$ is defined by $\phi_j(0)=0$ and
\begin{equation}\label{eq:def_phi_j}
\phi_j^\sigma(\ell)=(\Phi_j\circ\sigma_j^{-1})(\ell),\quad\text{for }\ell=1,\ldots,n_j,
\end{equation}
where
\begin{align*}
\sigma_j^{\!-\!1}(\ell)&=\min\Set{k:\sigma_j(k)=\ell\,}\quad\text{for }\ell=1,\ldots,n_j,\\
\Phi_j(k)\,&=\xi_k+\!\!\!\!\!\sum_{\ell<k\,:\,j_\ell\neq j_{\ell\!+\!1}}\!\!\!\!\!(-\!1)^{j+j_\ell}\,\xi_\ell \quad\text{for }k=1,\ldots,n.
\end{align*}
\end{lemma}

\begin{proof}
It is enough to show that the matrix $M$ can be obtained from $M'$ through the Gauss elimination. For each $k=1,\ldots,n$, replace one of the rows 
\begin{align*}
p_{\sigma_1(k)}+\phi_1^\sigma(\sigma_1(k))z,\quad &(\text{if }j_k=1) \\
q_{\sigma_2(k)}+\phi_2^\sigma(\sigma_2(k))z,\quad &(\text{if }j_k=2)
\end{align*}
by the row 
\begin{equation*}
p_{\sigma_1(k)}+q_{\sigma_2(k)}+\brS{\phi_1^\sigma(\sigma_1(k))+\phi_2^\sigma(\sigma_2(k))}z.
\end{equation*}
We claim that $\phi_1^\sigma(\sigma_1(k))+\phi_2^\sigma(\sigma_2(k))=\xi_k$ for each $k=1,\ldots,n$; if it is true, then every row of $M$ can be obtained by the above replacement of rows in $M'$. First, consider the case that either $\sigma_1(k)$ or $\sigma_2(k)$ is zero; assume $\sigma_2(k)=0$. This case implies $j_1=\cdots=j_k=1$ and $\sigma_1^{-\!1}(k)=k$. So, $$\phi_1^\sigma(\sigma_1(k))+\phi_2^\sigma(\sigma_2(k))=\phi_1^\sigma(k)+\phi_2^\sigma(0)=\Phi_1(k)=\xi_k.$$
Assume that both $\sigma_1(k)$ and $\sigma_2(k)$ are positive integers. Then, for $j=1,2$,
\begin{align*}
\phi_j^\sigma(\sigma_j(k))&=\Phi_j(\kappa\!_j) =\xi_{\kappa\!_j}+\!\!\!\!\!\sum_{\ell<\kappa\!_j\,:\,j_\ell\neq j_{\ell\!+\!1}}\!\!\!\!\!(-\!1)^{j+j_\ell}\,\xi_\ell, \quad\text{for }\kappa\!_j=\sigma_j^{\!-\!1}(k) \\
&= \xi_{\kappa\!_j} +(-1)^j\Big(\!\sum_{I_{\!12}^\sigma\cap[1,\kappa\!_j)}\xi_\ell - \sum_{I_{\!21}^\sigma\cap[1,\kappa\!_j)}\xi_\ell\Big),
\end{align*}
where $I_{\!12}^\sigma=\set{\ell:\, j_\ell=1,\, j_{\ell+1}=2}$ and $I_{\!21}^\sigma=\set{\ell:\, j_\ell=2,\, j_{\ell+1}=1}$.
If $j_k=1$, then $\kappa_1=\sigma_1^{-\!1}(k)=k$ and $j_{\kappa_2}=2$, $j_{\kappa_2+1}=\cdots=j_k=1$ for $\kappa_2=\sigma_2^{-\!1}(k)$. Thus, $I_{\!12}^\sigma\cap[1,\kappa_2)=I_{\!12}^\sigma\cap[1,k)$ and $\set{\kappa_2}\cup I_{\!21}^\sigma\cap[1,\kappa_2)=I_{\!21}^\sigma\cap[1,k)$. In case of $j_k=1$, we have
\begin{align}
\phi_1^\sigma(\sigma_1(k)) &=\xi_k +\sum_{I_{\!12}^\sigma\cap[1,k)}\xi_\ell - \sum_{I_{\!21}^\sigma\cap[1,k)}\xi_\ell, \label{eq:phi_sigma_jk1_1}\\
\phi_2^\sigma(\sigma_2(k)) &=\sum_{I_{\!21}^\sigma\cap[1,k)}\xi_\ell - \sum_{I_{\!12}^\sigma\cap[1,k)}\xi_\ell.  \label{eq:phi_sigma_jk1_2}
\end{align}
In case of $j_k=2$, we get $\kappa_2=\sigma_2^{-\!1}(k)=k$ and $j_{\kappa_1}=1$, $j_{\kappa_1+1}=\cdots=j_k=2$ for $\kappa_1=\sigma_1^{-\!1}(k)$. Thus, 
\begin{align}
\phi_1^\sigma(\sigma_1(k)) &=\sum_{I_{\!12}^\sigma\cap[1,k)}\xi_\ell - \sum_{I_{\!21}^\sigma\cap[1,k)}\xi_\ell, \label{eq:phi_sigma_jk2_1}\\
\phi_2^\sigma(\sigma_2(k)) &=\xi_k +\sum_{I_{\!21}^\sigma\cap[1,k)}\xi_\ell - \sum_{I_{\!12}^\sigma\cap[1,k)}\xi_\ell.  \label{eq:phi_sigma_jk2_2}
\end{align}
In both cases, we have $\phi_1^\sigma(\sigma_1(k))+\phi_2^\sigma(\sigma_2(k))=\xi_k$, which completes the proof.
\end{proof}

\begin{proposition}\label{prop:equal_volumes}
Let $H$ be a Hanner polytope in $\R^n$. Consider the set $C=(c\sub{F})$ of the centroids of faces of $H$. Then, the simplex $C_\F$, defined in \eqref{def_simplex}, has the same volume for every flag $\F$ of $H$.
\end{proposition}


\begin{proof}
Let $H$ be the $\ell_1$ or $\ell_\infty$-sum of two Hanner polytopes $H_1\subset\R^{n_1}$ and $H_2\subset\R^{n_2}$. Since the volume of the linear image of $C_\F$ by $T\in{\rm GL}(n)$ is always equal to $\abs{\det T}\abs{C_\F}$ for each flag $\F$, we may assume that $\R^n=\R^{n_1}\oplus\R^{n_2}$ is the orthogonal sum of $\R^{n_1}$ and $\R^{n_2}$.  

Let $\F=\set{F^0,\ldots,F^{n-1}}$ be a flag over $H$. Consider the lower dimensional flags, defined in \eqref{lower_flags}, $$\F_1=\set{F_1^0,\ldots,F_1^{n_1-1}}\quad\text{and}\quad\F_2=\set{F_2^0,\ldots,F_2^{n_2-1}}.$$
Let $\sigma=(j_1,\ldots,j_n)\in J$ be the type of $\F$ (Definition \ref{def:type}). For $j=1,2$, denote by $\sigma_j(k)$ the number of $j$'s among the first $k$ entries $j_1,\ldots,j_k$ of $\sigma$, and let $F_j^{-1}=\varnothing$, $F_j^{n_j}=H_j$ be the improper faces of $H_j$. \\

\noindent{\bf CASE ($H=H_1\oplus_1 H_2$)}.\quad It follows from Remark \ref{rmk:flag_form} that each element of a flag $\F=\set{F^0,\ldots,F^{n-1}}$ of type $\sigma$ is expressed by
 $$F^{k-1}=F_1^{\sigma_1(k)-1}\oplus_1 F_2^{\sigma_2(k)-1},\quad\text{for }1\le k\le n.$$
By Lemma \ref{centroid_formula}, the centroid of $F^{k-1}$ is given by
\begin{align*}
c\sub{F^{k-1}}&=\frac{\sigma_1(k)}{\sigma_1(k)+\sigma_2(k)}\,\,c_{\!F_1^{\sigma_1\!(k)\!-\!1}} +\frac{\sigma_2(k)}{\sigma_1(k)+\sigma_2(k)}\,\,c_{\!F_2^{\sigma_2\!(k)\!-\!1}}\\
&=\frac1k\brR{\sigma_1(k)\,c_{\!F_1^{\sigma_1\!(k)\!-\!1}}+\sigma_2(k)\,c_{\!F_2^{\sigma_2\!(k)\!-\!1}}}.
\end{align*}
Thus, the volume of the simplex $C_\F$ is equal to the absolute value of
\begin{align*}
\frac{1}{n!}\det\brR{\begin{array}{c}
c\sub{F^0}\\ \vdots\\ 
c\sub{F^{k-1}}\\ \vdots\\
c\sub{F^{n-1}}\\ 
\end{array}}
&=\frac{1}{n!^2}\det\brR{\begin{array}{c}
\vdots\\ \sigma_1(k)\,c_{\!F_1^{\sigma_1\!(k)\!-\!1}}+\sigma_2(k)\,c_{\!F_2^{\sigma_2\!(k)\!-\!1}}\\ \vdots\\
\end{array}}.
\end{align*}
Applying Lemma \ref{lem:gauss_elimination} with $\xi_1=\cdots=\xi_n=0$, the absolute value of the determinant of the matrix whose rows are $$\sigma_1(k)\,c_{\!F_1^{\sigma_1(k)-1}}+\sigma_2(k)\,c_{\!F_2^{\sigma_2(k)-1}},\quad k=1,\ldots,n$$ is equal to that of the matrix whose rows are
\begin{align*}
k_1\,c_{\!F_1^{k_1\!-\!1}}, \qquad &k_1=1,\ldots,n_1\\ 
k_2\,c_{\!F_2^{k_2\!-\!1}}, \qquad &k_2=1,\ldots,n_2.
\end{align*}
The determinant of the above matrix is, up to a multiple of $\pm1$, equal to
\begin{align*}
\det\brR{\begin{array}{c|c}
1\cdot c\sub{F_1^0} & 0\\ \vdots & \vdots\\
n_1\cdot c_{\!F_1^{n_1\!-\!1}} & 0\\ \hline
0 & 1\cdot c\sub{F_2^0}\\ \vdots & \vdots\\
0 & n_2\cdot c_{\!F_2^{n_2\!-\!1}}
\end{array}}
=(n_1!\,n_2!)\det\brR{\begin{array}{c}
c\sub{F_1^0}\\ \vdots\\
c_{\!F_1^{n_1\!-\!1}}
\end{array}}\det\brR{\begin{array}{c}
c\sub{F_2^0}\\ \vdots\\
c_{\!F_2^{n_2\!-\!1}}
\end{array}},
\end{align*}
Thus, for the $\ell_1$-sum case, we have
\begin{equation}\label{ell1_vol_formula}
\abs{C_\F}=\brR{\frac{n_1!n_2!}{n!}}^2\Abs{(C_1)_{\F_1}}\cdot\Abs{(C_2)_{\F_2}}
\end{equation}
where each $C_j=(c\sub{F_j})$ is the set of centroids of faces of $H_j$ for $j=1,2$.\\

\noindent{\bf CASE ($H=H_1\oplus_\infty H_2$)}.\quad From Remark \ref{rmk:flag_form}, each element of the flag $\F=\set{F^0,\ldots,F^{n-1}}$ of type $\sigma$ can be expressed by $$F^{n-k}=F_1^{n_1-\sigma_1(k)}\oplus F_2^{n_2-\sigma_2(k)}.$$ 
By Lemma \ref{centroid_formula}, the volume of the simplex $C_\F$ is equal to the absolute value of
\begin{align*}
\frac1{n!}\det\brR{\begin{array}{c}
c\sub{F^0}\\ \vdots\\ 
c\sub{F^{n-k}}\\ \vdots\\
c\sub{F^{n-1}}\\ 
\end{array}}
&=\frac1{n!}\det\brR{\begin{array}{c}
\vdots\\ c_{\!F_1^{n_1-\sigma_1(k)}}+c_{\!F_2^{n_2-\sigma_2(k)}}\\ \vdots\\
\end{array}}.
\end{align*}
Applying Lemma \ref{lem:gauss_elimination} with $\xi_1=\cdots=\xi_n=0$, the absolute value of  the determinant of the matrix whose rows are $$c_{\!F_1^{n_1-\sigma_1(k)}}+c_{\!F_2^{n_2-\sigma_2(k)}},\quad k=1,\ldots,n$$  is equal to that of the matrix whose rows are
\begin{align*}
c_{\!F_1^{n_1-k_1}}, \qquad &k_1=1,\ldots,n_1\\ 
c_{\!F_2^{n_2-k_2}}, \qquad &k_2=1,\ldots,n_2.
\end{align*}
Then, by a similar argument to the $\ell_1$-sum case, the volume of $C_\F$ for the $\ell_\infty$-sum case is given by
\begin{align}\label{ell_inf_vol_formula}
\abs{C_\F}
&=\frac{n_1!n_2!}{n!}\Abs{(C_1)_{\F_1}}\cdot\Abs{(C_2)_{\F_2}},
\end{align}
where each $C_j=(c\sub{F_j})$ is the set of centroids of faces of $H_j$ for $j=1,2$.

Finally, we use the induction on the dimension of $H$ to conclude from \eqref{ell1_vol_formula}, \eqref{ell_inf_vol_formula} that the simplices $C_\F$ have equal volumes for each flag $\F$ over $H$. 

\end{proof}

\begin{corollary}\label{cor:differentiable}
Let $H$ be a Hanner polytope in $\R^n$ and $C=(c\sub{F})$ the set of centroids of faces of $H$. Then, the volume function $V(\cdot)$ for $H$ given in Definition \ref{def:vol_function} is infinitely differentiable around the point $C=(c\sub{F})$.
\end{corollary}

Note that the function $V(\cdot)$ is a sum of the volumes of simplices, and the volume of each simplex can be expressed by the absolute value of a determinant function. Since every simplex $C_\F$ has equal volume which is non-zero, the volume function $V(\cdot)$ around the point $C=(c\sub{F})$ is just a sum of the determinant functions. Thus $V(\cdot)$ is infinitely differentiable in a neighborhood of the point $C=(c\sub{F})$.

\section{Differential of the Volume Function $V(\cdot)$}\label{differential_volume_function}

In this section we prove that the condition \eqref{differential_zero} in Proposition \ref{prop:vol_difference}, $$\brA{V'(C),Z}=0 \quad\text{for any }Z=\big(z_F\big)\text{ with all }z\sub{F}\in A\sub{F}-A\sub{F}$$ holds for Hanner polytopes. By linearity it is enough to take $Z=\big(z\sub{F}\big)$ with 
\begin{equation*}
z\sub{F} = \begin{cases} z, & \text{if $F=G$}\\ 
0, & \text{otherwise} 
\end{cases}
\end{equation*}
for a fixed face $G$ and $z\in A_{G}-A_{G}$. Since
\begin{align*}
\brA{V'(C),Z}&=\lim_{t\rightarrow0}{V(C+tZ)-V(C)\over t} \\
&= \lim_{t\rightarrow0}\frac1t\brR{\sum_{\F\ni G}\abs{(C+tZ)_\F}-\sum_{\F\ni G}\abs{C_\F}},
\end{align*}
we need to show $\sum\limits_{\F\ni G}\abs{(C+tZ)_\F}=\sum\limits_{\F\ni G}\abs{C_\F}$ for small $t>0$ to get $\brA{V'(C),Z}=0$.

\begin{lemma}\label{lem:stability_for_even}
Let $H$ be a Hanner polytope in $\R^n$ and $C=(c\sub{F})$ the set of centroids of faces of $H$. For $z$, $\xi=(\xi_1,\ldots,\xi_n)\in\R^n$, consider the point $W_{\xi,z}=(w\sub{F})$ defined by $w\sub{F}=\xi_{\dim F+1}\,z$ for each face $F$ of $H$. Then, for any $\xi$ close to the origin,  $V(C+W_{\xi,z})=V(C)$, i.e., $$\sum_\F\abs{(C+W_{\xi,z})_\F}=\sum_\F\abs{C_\F}.$$
\end{lemma}

\begin{proof}
Let $z=(z_1,\ldots,z_n)\in\R^n$.
Consider the function from $\R^n$ to $\R_+$, defined by 
\begin{equation}\label{eq:det_affinemap}
\xi=(\xi_1,\ldots,\xi_n)\longmapsto \det(M+\xi^Tz),
\end{equation}
where $M$ is a $n\times n$ matrix and $\xi^Tz$ is the $n\times n$ matrix with $\xi_i z_j$ in the $(i,j)$-entry. Note that $\frac{\partial(\det(M+\xi^Tz))}{\partial\xi_{i_1}\cdots\partial\xi_{i_k}}$ is equal to the determinant of the matrix obtained from $M$ by replacing the $i_1,\ldots,i_k$-th rows of $M$ with the same row $z$. It follows that all second-order derivatives are zero, so the function \eqref{eq:det_affinemap} is affine. Moreover, since the volume function $V(C+W_{\xi,z})$ can be expressed by a sum of determinant functions similar to \eqref{eq:det_affinemap}, the function 
\begin{equation}\label{eq:det_evenmap}
\xi\longmapsto V(C+W_{\xi,z})
\end{equation}
is also affine. Also, \eqref{eq:det_evenmap} is an even function. Indeed,
\begin{align*}
V(C+W_{-\xi,z})&=\sum_\F \abs{(C-W_{\xi,z})_\F}=\sum_\F \abs{(C-W_{\xi,z})_{-\F}}\\
&=\sum_\F \abs{(-C-W_{\xi,z})_\F}=V(C+W_{\xi,z}). 
\end{align*}
Thus, it should be a constant. It implies that $V(C+W_{\xi,z})=V(C)$ when $\xi$ is close to the origin. 
\end{proof}

\begin{theorem}\label{thm:differential_zero}
Let $H$ be a Hanner polytope in $\R^n$ and $C=(c\sub{F})$ the set of centroids of faces of $H$. Fix a face $G$ of $H$. Take any $\xi_1,\ldots,\xi_n\in\R$ with small absolute values, and $z\in A_G-A_G$ where $A_G$ is the affine subspace defined in Section \ref{Hanner_polytopes}. Consider $W=(w\sub{F})$ with
\begin{equation*}
w\sub{F} = \begin{cases} \xi_{\dim F+1}\,z, & \text{if $F\supset G$ or $F\subset G$}\\ 
0, & \text{otherwise}.
\end{cases}
\end{equation*}
Then $$\sum_{\F\ni G}\abs{(C+W)_\F}=\sum_{\F\ni G}\abs{C_\F}.$$
\end{theorem}

\begin{proof}
Let $H$ be the $\ell_1$ or $\ell_\infty$-sum of two Hanner polytopes $H_1\subset\R^{n_1}$ and $H_2\subset\R^{n_2}$. Assume that $\R^n=\R^{n_1}\oplus\R^{n_2}$ is the orthogonal sum of $\R^{n_1}$ and $\R^{n_2}$  Fix a face $G$ of $H$, and consider a flag $\F=\set{F^0,\ldots,F^{n-1}}$ over $H$ containing $G$. Then $\F$ induces two lower dimensional flags $\F_1=\{F_1^0,\ldots,F_1^{n_1-1}\}$, $\F_2=\{F_2^0,\ldots,F_2^{n_2-1}\}$ defined in \eqref{lower_flags}. Denote by $F_j^{-1}=\varnothing$ and $F_j^{n_j-1}=H_j$ the improper faces of $H_j$  for each $j=1,2$. Let 
\begin{equation*}
m=\begin{cases}
1+\dim G, &\text{for the $\ell_1$-sum case}\\
n-\dim G, &\text{for the $\ell_\infty$-sum case}.
\end{cases}
\end{equation*}
From \eqref{face_form1}, \eqref{face_form2}, the face $G$ can be written by
\begin{equation}\label{eq:decomposition_G}
G=\begin{cases}
F^{m-1}=F_1^{m_1-1}\oplus_1 F_2^{m_2-1}, & \text{for the $\ell_1$-sum case}\\
F^{n-m}=F_1^{n_1-m_1}\oplus F_2^{n_2-m_2}, & \text{for the $\ell_\infty$-sum case}
\end{cases}
\end{equation}
where $m_1$, $m_2$ are integers satisfying $0\le m_1\le n_1$ $0\le m_2\le n_2$ and $m_1+m_2=m$. Suppose that $\F$ is of type $\sigma=(j_1,\ldots,j_n)$. Since $\F$ contains $G$, it follows from the above representation of $G$ and Remark \ref{rmk:flag_form} that the number of $j$'s among the first $m$ entries of $\sigma$ is $m_j$ for each $j=1,2$. Thus, the type $\sigma$ of a flag $\F$ containing the face $G$ can be viewed as an element of 
\begin{equation}\label{def_Jm}
J_m=\set{\sigma=(j_1,\ldots,j_n)\in\set{1,2}^n\Big|\,\,\substack{\sigma_1(n)=n_1,\\ \sigma_2(n)=n_2,}\,\substack{\sigma_1(m)=m_1\\ \sigma_2(m)=m_2}\,}
\end{equation}
where $\sigma_j(k)$ denotes the number of $j$'s among the first $k$ entries $j_1,\ldots,j_k$ of $\sigma=(j_1,\ldots,j_n)$ for $j=1,2$. As in Remark \ref{rmk:flag_form}, every flag over the $\ell_1$ or $\ell_\infty$ sum of $H_1$, $H_2$ containing the face $G$ is associated with an triple consisting of an element of $J_m$ and two lower dimensional flags over $H_1$, $H_2$ containing the corresponding summands in the decomposition \eqref{eq:decomposition_G} of $G$.\\

\noindent{\bf CASE ($H=H_1\oplus_1 H_2$)}.\quad It follows from Remark \ref{rmk:flag_form} that each element of a flag $\F=\set{F^0,\ldots,F^{n-1}}$ of type $\sigma$ is expressed by $$F^{k-1}=F_1^{\sigma_1(k)-1}\oplus_1 F_2^{\sigma_2(k)-1},\quad(1\le k\le n).$$
By Lemma \ref{centroid_formula}, the centroid of $F^{k-1}$ is given by
$$c\sub{F^{k\!-\!1}}=\frac1k\brR{\sigma_1(k)\,c_{\!F_1^{\sigma_1\!(k)\!-\!1}}+\sigma_2(k)\,c_{\!F_2^{\sigma_2\!(k)\!-\!1}}}.$$
So, the volume of the simplex $(C+W)_\F$ is equal to the absolute value of
\begin{align*}
\frac1{n!}\det\brR{\begin{array}{c}
c\sub{F^0}+\xi_1 z\\ \vdots\\ 
c\sub{F^{k-1}}+\xi_k z\\ \vdots\\
c\sub{F^{n-1}}+\xi_n z\\ 
\end{array}}
&=\frac1{n!^2}\det\brR{\begin{array}{c}
\vdots\\ \sigma_1(k)\,c_{\!F_1^{\sigma_1\!(k)\!-\!1}}+\sigma_2(k)\,c_{\!F_2^{\sigma_2\!(k)\!-\!1}} + (k\xi_k) z\\ \vdots\\
\end{array}}.
\end{align*}
By Lemma \ref{lem:gauss_elimination}, the absolute value of the determinant of the matrix whose rows are $$\sigma_1(k)\,c_{\!F_1^{\sigma_1\!(k)\!-\!1}}+\sigma_2(k)\,c_{\!F_2^{\sigma_2\!(k)\!-\!1}} + (k\xi_k) z,\quad k=1,\ldots,n$$ is equal to that of the matrix whose rows are
\begin{align*}
k_1\,c_{\!F_1^{k_1-1}} + k_1\bar\phi_1(k_1) z, \quad &k_1=1,\ldots,n_1\\ 
k_2\,c_{\!F_2^{k_2-1}} + k_2\bar\phi_2(k_2) z, \quad &k_2=1,\ldots,n_2
\end{align*}
where $\bar\phi_1$, $\bar\phi_2$ are some functions depending on $\sigma\in J_m$ and $\xi_1,\ldots,\xi_n$. The exact formulas for $\bar\phi_1$, $\bar\phi_2$ are not needed in this proof of the $\ell_1$-sum case. Thus, the volume of the simplex $(C+W)_\F$ is equal to the absolute value of 
\begin{equation}\label{eq:volume_simplex_CW_F}
\frac1{n!^2}\det\brR{\begin{array}{c}
\vdots\\ k_1\,c_{\!F_1^{k_1-1}} + k_1\bar\phi_1(k_1) z\\ 
k_2\,c_{\!F_2^{k_2-1}} + k_2\bar\phi_2(k_2) z\\ \vdots\\
\end{array}}=\frac{n_1!\,n_2!}{n!^2}\det\brR{\begin{array}{c}
\vdots\\ c_{\!F_1^{k_1-1}} + \bar\phi_1(k_1) z\\ 
c_{\!F_2^{k_2-1}} + \bar\phi_2(k_2) z\\ \vdots\\
\end{array}}.
\end{equation}
Note that $z\in A_G-A_G$, and $G$ is written from \eqref{eq:decomposition_G} by 
\begin{equation*}
G=G_1\oplus_1G_2\quad\text{for }\, G_1=F_1^{m_1-1}\text{ and } G_2=F_2^{m_2-1}.
\end{equation*}
Also, from \eqref{def_AF1} we have 
\begin{equation*}
A_{G}-A_{G}=
\begin{cases}
(A_{G_1}-A_{G_1}) + (A_{G_2}-A_{G_2}), &\text{if }G_1\neq\varnothing, G_2\neq\varnothing\\
(A_{G_1}-A_{G_1}) + \R^{n_2}, &\text{if }G_1\neq\varnothing, G_2=\varnothing \\
\,\,\R^{n_1} + (A_{G_2}-A_{G_2}), &\text{if }G_1=\varnothing, G_2\neq\varnothing.
\end{cases}
\end{equation*}
Without loss of generality, we may assume one of the following three cases.
\begin{enumerate}
\item $z\in A_{G_1}-A_{G_1}$ when $G_1\neq\varnothing$, $G_2\neq\varnothing$.
\item $z\in A_{G_1}-A_{G_1}$ when $G_1\neq\varnothing$, $G_2=\varnothing$.
\item $z\in \R^{n_1}$\quad when $G_1=\varnothing$, $G_2\neq\varnothing$.
\end{enumerate}
Since $z\in\R^{n_1}$ in any cases, from \eqref{eq:volume_simplex_CW_F} we have
\begin{align}\label{differential_vol1}
\abs{(C+W)_\F} &=\frac{n_1! n_2!}{n!^2}\abs{\det\brR{\begin{array}{ccc|c}
c\sub{F_1^0}&+&\bar\phi_1 (1)z &0 \\
&\vdots&&\vdots \\
c_{\!F_1^{n_1\!-\!1}}&+&\bar\phi_1 (n_1)z  &0 \\
\hline
&&\bar\phi_2 (1)z & c\sub{F_2^0}\\
&&\vdots&\vdots \\
&&\bar\phi_2 (n_2)z & c_{\!F_2^{n_2-1}}
\end{array}}} \notag \\
&=\brR{\frac{n_1!n_2!}{n!}}^2\Abs{(C_1+W_\sigma)_{\F_1}}\cdot\Abs{(C_2)_{\F_2}}
\end{align}
where $W_\sigma=(w\sub{F})$ is defined by 
\begin{equation*}
w\sub{F} = \begin{cases} \bar\phi_1 (\dim F+1)\, z, & \text{if $F\supset G_1$ or $F\subset G_1$}\\ 
0, & \text{otherwise.} 
\end{cases}
\end{equation*}
Consider the first case that $z\in A_{G_1}-A_{G_1}$ when $G_1\neq\varnothing$ and $G_2\neq\varnothing$. Since $G=G_1\oplus_1G_2$,  
\begin{equation*}
\sum_{\F\ni G}\abs{(C+W)_\F}=\sum_{\sigma\in J_m}\sum_{\F_1\ni G_1}\,\sum_{\F_2\ni G_2}\abs{(C+W)_\F},
\end{equation*}
where the flag $\F$ on the right hand side means the flag induced by lower dimensional flags $\F_1$, $\F_2$ and a type $\sigma$ as in Remark \ref{rmk:flag_form}.
It follows from \eqref{differential_vol1} that
\begin{align}\label{eq:C+W_LHS}
\sum_{\F\ni G}\abs{(C+W)_\F}&={\rm (Const)}\sum_{\sigma\in J_m}\,\,\sum_{\F_1\ni G_1}\abs{(C_1+W_\sigma)_{\F_1}},
\end{align}
where ${\rm(Const)}=\brR{\frac{n_1!n_2!}{n!}}^2\sum_{\F_2\ni G_2}\abs{(C_2)_{\F_2}}$. Similarly, letting $\xi_1=\cdots=\xi_n=0$ we get 
\begin{equation}\label{eq:C_RHS}
\sum_{\F\ni G}\abs{C_\F}={\rm (Const)}\sum_{\sigma\in J_m}\,\,\sum_{\F_1\ni G_1}\abs{(C_1)_{\F_1}}.
\end{equation}
Since $$\sum_{\F_1\ni G_1}\abs{(C_1+W_\sigma)_{\F_1}}=\sum_{\F_1\ni G_1}\abs{(C_1)_{\F_1}}$$ by the induction hypothesis, the right hand side of \eqref{eq:C+W_LHS} is the same as that of \eqref{eq:C_RHS}, which implies $\sum_{\F\ni G}\abs{(C+W)_\F}=\sum_{\F\ni G}\abs{C_\F}$. Similarly, in the second case, we have the same conclusion with a different constant  ${\rm(Const)}=\brR{\frac{n_1!n_2!}{n!}}^2\sum_{\F_2}\abs{(C_2)_{\F_2}}$ for \eqref{eq:C+W_LHS}.

Consider the third case that $z\in \R^{n_1}$ when $G_1=\varnothing$ and $G_2\neq\varnothing$. In this case, note that $W_\sigma=(w\sub{F})$ is the point with $w\sub{F}=\bar\phi_1 (\dim F+1)\, z$ for each face $F$ of $H_1$ (without any other restrictions on $F$). So, Lemma \ref{lem:stability_for_even} gives $$\sum_{\F_1}\abs{(C+W_\sigma)_\F}=\sum_{\F_1}\abs{C_\F}.$$ Therefore,
\begin{align*}
\sum_{\F\ni G}\abs{(C+W)_\F}&={\rm (Const)}\sum_{\sigma\in J_m}\sum_{\F_1}\abs{(C+W_\sigma)_\F} ={\rm (Const)}\sum_{\sigma\in J_m}\sum_{\F_1}\abs{C_\F}=\sum_{\F\ni G}\abs{C_\F},
\end{align*}
which completes the proof for the $\ell_1$-sum case.\\

\noindent{\bf CASE ($H=H_1\oplus_\infty H_2$)}.\quad From Remark \ref{rmk:flag_form}, each element of the flag $\F=\set{F^0,\ldots,F^{n-1}}$ of type $\sigma$ can be expressed by $$F^{n-k}=F_1^{n_1-\sigma_1(k)}\oplus F_2^{n_2-\sigma_2(k)}.$$
Note that $z\in A_G-A_G$, and by \eqref{eq:decomposition_G} 
\begin{equation*}
G=G_1\oplus G_2\quad\text{for }\, G_1=F_1^{n_1-m_1}\text{ and } G_2=F_2^{n_2-m_2}.
\end{equation*}
In addition, $A_{G}-A_{G}$ can be written from \eqref{def_AF2} by
\begin{align*}
A_{G_1}-A_{G_1},\quad &\text{if }G_1\neq H_1, G_2=H_2 \\
A_{G_2}-A_{G_2},\quad &\text{if }G_1=H_1, G_2\neq H_2 
\end{align*}
and
\begin{equation*}
(A_{G_1}-A_{G_1}) + (A_{G_2}-A_{G_2}) + \spn\set{\frac1{m_1}c^{}_{G_1} \!-\! \frac1{m_2}c^{}_{G_2}}\quad \text{if }G_1\neq H_1, G_2\neq H_2.
\end{equation*}
Without loss of generality, we may consider the following three cases: 
\begin{enumerate}
\item $z\in A_{G_1}-A_{G_1}$ when $G_1\neq H_1$, $G_2=H_2$.
\item $z\in A_{G_1}-A_{G_1}$ when $G_1\neq H_1$, $G_2\neq H_2$.
\item $z\in\spn\set{\frac1{m_1}c^{}_{G_1} \!-\! \frac1{m_2}c^{}_{G_2}}$ when $G_1\neq H_1$, $G_2\neq H_2$.
\end{enumerate}
The volume of the simplex $(C+W)_\F$ is, up to a multiple of $\pm1$,
\begin{align*}
\frac{1}{n!}\det\brR{\begin{array}{c}
c\sub{F^0}+\xi_1 z\\ \vdots\\ 
c\sub{F^{n-k}}+\xi_k z\\ \vdots\\
c\sub{F^{n-1}}+\xi_n z\\ 
\end{array}}
&=\frac{1}{n!}\det\brR{\begin{array}{c}
\vdots\\ c_{\!F_1^{n_1-\sigma_1(k)}}+c_{\!F_2^{n_2-\sigma_2(k)}} + \xi_k z\\ \vdots\\
\end{array}}.
\end{align*}
Here, for simplicity, $\xi_1,\ldots,\xi_n$ are rearranged by $\xi_{n-\dim F}$ instead of $\xi_{\dim F+1}$.
By Lemma \ref{lem:gauss_elimination}, the determinant of the matrix whose rows are $$c_{\!F_1^{n_1-\sigma_1(k)}}+c_{\!F_2^{n_2-\sigma_2(k)}} + \xi_k z,\quad k=1,\ldots,n$$ is equal to that of the matrix whose rows are
\begin{align*}
c_{\!F_1^{k_1-1}} + \phi_1^\sigma(k_1) z, \quad &k_1=1,\ldots,n_1\\ 
c_{\!F_2^{k_2-1}} + \phi_2^\sigma(k_2) z, \quad &k_2=1,\ldots,n_2.
\end{align*}
where $\phi_1^\sigma$, $\phi_2^\sigma$ are the functions defined in \eqref{eq:def_phi_j}.
Thus
\begin{align}\label{differential_vol2}
\abs{(C+W)_\F} &=\frac1{n!}\abs{\det\brR{\begin{array}{c}
\vdots\\ c_{\!F_1^{n_1-k_1}} + \phi_1^\sigma(k_1) z\\ 
c_{\!F_2^{n_2-k_2}} + \phi_2^\sigma(k_2) z\\ \vdots\\
\end{array}}}.
\end{align}
For the first two cases $z\in A_{G_1}-A_{G_1}$ when $G_1\neq H_1$, we can conclude by the same argument as in the $\ell_1$-sum case to get $\sum_{\F\ni G}\abs{(C+W)_\F}=\sum_{\F\ni G}\abs{C_\F}$. For the third one, let $z=\frac1{m_1}c^{}_{G_1} \!-\! \frac1{m_2}c^{}_{G_2}$. In this case, the $2\times n$ sub-matrix from \eqref{differential_vol2}
\begin{align*}
\brR{\begin{array}{c}
c^{}_{G_1} + \phi_1^\sigma(m_1) z\\ 
c^{}_{G_2} + \phi_2^\sigma(m_2) z
\end{array}}
&=\brR{\begin{array}{rcr}
\brS{1+\frac{\phi_1^\sigma(m_1)}{m_1}}c^{}_{G_1} &-& \frac{\phi_1^\sigma(m_1)}{m_2}\,c^{}_{G_2} \\ 
\frac{\phi_2^\sigma(m_2)}{m_1}\,c^{}_{G_1} &+& \brS{1-\frac{\phi_2^\sigma(m_2)}{m_2}}c^{}_{G_2}  
\end{array}}
\end{align*}
is, under the Gauss elimination,  equivalent to 
\begin{align*}
\brR{\begin{array}{c}
\la_1\,c^{}_{G_1}\\ 
\la_2\,c^{}_{G_2}
\end{array}}
\end{align*}
where $\la_1=1+\frac{\phi_1^\sigma(m_1)}{m_1}$ and $\la_2=1-\frac{\phi_2^\sigma(m_2)}{m_2}+\frac{\phi_1^\sigma(m_1)}{m_2}\frac{\phi_2^\sigma(m_2)}{m_1}\big[1+\frac{\phi_1^\sigma(m_1)}{m_1}\big]^{-1}$. So, $$\la_1\la_2=1+\frac{\phi_1^\sigma(m_1)}{m_1}-\frac{\phi_2^\sigma(m_2)}{m_2}.$$
Since $z$ is a linear combination of two rows $c^{}_{G_1}$ and $c^{}_{G_2}$, the $z$-term in each row of the matrix in \eqref{differential_vol2} disappears through the Gauss elimination, that is, the absolute value of
\begin{align*}
\det\brR{\begin{array}{c}
\vdots\\ c_{\!F_1^{n_1-k_1}} + \phi_1^\sigma(k_1) z\\ 
c_{\!F_2^{n_2-k_2}} + \phi_2^\sigma(k_2) z\\ \vdots\\
\end{array}}
=\brR{1+\frac{\phi_1^\sigma(m_1)}{m_1} - \frac{\phi_2^\sigma(m_2)}{m_2}} \cdot
\det\brR{\begin{array}{c}
\vdots\\ c\sub{F_1^{n_1-k_1}} + \phi_1^\sigma(k_1) z\\ 
c_{\!F_2^{n_2-k_2}} + \phi_2^\sigma(k_2) z\\ \vdots\\
c^{}_{G_1}\\ 
c^{}_{G_2}\\
\vdots\\ c_{\!F_1^{n_1-k_1'}} + \phi_1^\sigma(k_1') z\\ 
c_{\!F_2^{n_2-k_2'}} + \phi_2^\sigma(k_2') z\\ \vdots\\
\end{array}}
\end{align*}
is equal to
\begin{align}\label{differential_vol3}
\abs{(C+W)_\F} &=\brR{1+\frac{\phi_1^\sigma(m_1)}{m_1} - \frac{\phi_2^\sigma(m_2)}{m_2}}\cdot\frac1{n!}
\abs{\det\brR{\begin{array}{c}
c\sub{F_1^0}\\ \vdots\\  c\sub{F_1^{n_1\!-\!1}}
\end{array}}\cdot
\det\brR{\begin{array}{c}
c\sub{F_2^0}\\ \vdots\\  c\sub{F_2^{n_2\!-\!1}}
\end{array}}} \notag \\
&= \brR{1+\frac{\phi_1^\sigma(m_1)}{m_1} - \frac{\phi_2^\sigma(m_2)}{m_2}}\cdot\frac{n_1!\,n_2!}{n!}\abs{(C_1)_{\F_1}}\abs{(C_2)_{\F_2}}.
\end{align}
From $G=G_1\oplus G_2$, we get
\begin{align*}
\sum_{\F\ni G}\abs{(C+W)_\F}&=\sum_{\sigma\in J_m}\sum_{\F_1\ni G_1}\,\sum_{\F_2\ni G_2}\abs{(C+W)_\F},
\end{align*}
where the flag $\F$ on the right hand side is the flag induced by lower dimensional flags $\F_1$, $\F_2$ and a type $\sigma$ as in Remark \ref{rmk:flag_form}.
The formula \eqref{differential_vol3} implies
\begin{equation}\label{eq:RHS}
\sum_{\F\ni G}\abs{(C+W)_\F}={\rm (Const)}\sum_{\sigma\in J_m}\brR{1+\frac{\phi_1^\sigma(m_1)}{m_1} - \frac{\phi_2^\sigma(m_2)}{m_2}}
\end{equation}
where ${\rm(Const)}=\frac{n_1!n_2!}{n!}\sum_{\F_1\ni G_1}\sum_{\F_2\ni G_2}\abs{(C_1)_{\F_1}}\abs{(C_2)_{\F_2}}$. Similarly, letting $\xi_1=\cdots=\xi_n=0$, we get $\phi_1^\sigma(k)=0=\phi_2^\sigma(k)$ for each $k$ and hence
\begin{equation}\label{eq:LHS}
\sum_{\F\ni G}\abs{C_\F}={\rm (Const)}\sum_{\sigma\in J_m}1.
\end{equation}
To complete the proof from \eqref{eq:RHS} and \eqref{eq:LHS}, we claim that 
\begin{equation}\label{eq:sum_phi12_m12}
\sum_{\sigma\in J_m}\brR{\frac{\phi_1^\sigma(m_1)}{m_1} - \frac{\phi_2^\sigma(m_2)}{m_2}}=0.
\end{equation}
It follows from \eqref{eq:phi_sigma_jk1_1}, \eqref{eq:phi_sigma_jk1_2}, \eqref{eq:phi_sigma_jk2_1}, \eqref{eq:phi_sigma_jk2_2} that
\begin{equation*}
\frac{\phi_1^\sigma(m_1)}{m_1} - \frac{\phi_2^\sigma(m_2)}{m_2}=
\frac{(-1)^{1+j_m}}{m_{j_m}}\xi_m+\brR{\frac1{m_1}+\frac1{m_2}}\Big(\sum\limits_{I_{\!12}^\sigma\cap[1,m)}\xi_\ell - \sum\limits_{I_{\!21}^\sigma\cap[1,k)}\xi_\ell\Big),
\end{equation*}
where $I_{\!12}^\sigma=\set{\ell:\, j_\ell=1,\, j_{\ell+1}=2}$ and $I_{\!21}^\sigma=\set{\ell:\, j_\ell=2,\, j_{\ell+1}=1}$.
Therefore,
\begin{align*}
\sum_{\sigma\in J_m}\brR{\frac{\phi_1^\sigma(m_1)}{m_1} - \frac{\phi_2^\sigma(m_2)}{m_2}}=\brR{\frac{\mu_m^1}{m_1}-\frac{\mu_m^2}{m_2}}\xi_m + \brR{\frac1{m_1}+\frac1{m_2}}\sum_{\ell=1}^{m-1} (\mu_\ell^{12}-\mu_\ell^{21})\xi_\ell,
\end{align*}
where $\mu_m^1$, $\mu_m^2$ and $\mu_\ell^{12}$, $\mu_\ell^{21}$ for $\ell=1,\ldots,m-1$ are the numbers given by
\begin{align*}
\mu_m^1 &= \#\Set{\sigma=(j_1,\ldots,j_n)\in J_m:j_m=1} = \binom{m-1}{m_1-1}\cdot\binom{n-m}{n_1-m_1}\\
\mu_m^2 &= \#\Set{\sigma=(j_1,\ldots,j_n)\in J_m:j_m=2} = \binom{m-1}{m_2-1}\cdot\binom{n-m}{n_1-m_1}\\
\mu_\ell^{12} &= \#\Set{\sigma=(j_1,\ldots,j_n)\in J_m:j_\ell=1, j_{\ell+1}=2}= \binom{m-2}{m_1-1}\cdot\binom{n-m}{n_1-m_1}\\
\mu_\ell^{21} &= \#\Set{\sigma=(j_1,\ldots,j_n)\in J_m:j_\ell=2, j_{\ell+1}=1}= \binom{m-2}{m_2-1}\cdot\binom{n-m}{n_1-m_1}.
\end{align*}
We can easily see that $\frac{\mu_m^1}{m_1}=\frac{\mu_m^2}{m_2}$ and $\mu_\ell^{12}=\mu_\ell^{21}$ from $n_1+n_2=n$ and $m_1+m_2=m$, which implies \eqref{eq:sum_phi12_m12} and completes the proof for the $\ell_\infty$-sum case.
\end{proof}

\begin{corollary}\label{cor:differential_zero}
Let $H$ be a Hanner polytope in $\R^n$, $C=(c\sub{F})$ the set of centroids of faces of $H$, and $A\sub{F}$'s the affine subspaces defined from $H$ as in Section \ref{Hanner_polytopes}. Then $$\brA{V'(C),Z}=0 \quad\text{for any }Z=\big(z\sub{F}\big)\text{ with all }z\sub{F}\in A\sub{F}-A\sub{F}.$$
\end{corollary}

\begin{proof}
By linearity it is enough to take $Z=\big(z\sub{F}\big)$ with 
\begin{equation*}
z\sub{F} = \begin{cases} z, & \text{if $F=G$}\\ 
0, & \text{otherwise} 
\end{cases}
\end{equation*}
for a face $G$ and $z\in A_{G}-A_{G}$. Letting all $\xi_1,\ldots,\xi_n$ be zero except $\xi_{1+\dim G}$ in Theorem \ref{thm:differential_zero}, we have
\begin{align*}
\brA{V'(C),Z}&=\lim_{t\rightarrow0}{V(C+tZ)-V(C)\over t} = \lim_{t\rightarrow0}\frac1t\brR{\sum_{\F\ni G}\abs{(C+tZ)_\F}-\sum_{\F\ni G}\abs{C_\F}}=0,
\end{align*}
which completes the proof.
\end{proof}


\section{Distance between a body $K$ and a polytope }\label{delta_gap}

We recall that every Hanner polytope in $\R^n$ can be obtained from $n$ symmetric intervals in $\R^n$ by taking the $\ell_1$ or $\ell_\infty$ sums. In particular, a Hanner polytope in $\R^n$ is called \emph{standard} if it is obtained from the intervals $[-e_1,e_1],\ldots,[-e_n,e_n]$ by taking the $\ell_1$ or $\ell_\infty$ sums. It is easy to see that every Hanner polytope is a linear image of a standard Hanner polytope. Since the volume product of a symmetric convex body is invariant under linear transformations, we can start with fixing $H$ as a standard Hanner polytope for the proof of Main theorem.

It is known (e.g.\ see \cite{Re3} or \cite{KimL}) that every standard Hanner polytope $H$ in $\R^n$ can be associated with a graph $G$ with the vertex set $\set{1,\cdots,n}$ and the edge set defined as follows: two different points $i,j\in\set{1,\cdots,n}$ are connected by an edge of $G$ if $e_i+e_j$ does not belong to $H$. We write
\begin{align*}
i\sim j \quad&\text{if \,\,$e_i+e_j\notin H$,  $i\neq j$,}\\
i\nsim j \quad&\text{if \,\,$e_i+e_j\in H$,  $i\neq j$.}
\end{align*}
In fact, if $i$, $j$ are connected to each other by an edge of $G$, the section of $H$ by $\spn\set{e_i,e_j}$ is the unit ball of 2-dimensional $\ell_1$. Otherwise, the section should be the unit ball of 2-dimensional $\ell_\infty$. Thus, the graph associated with the cross-polytope $B_1^n$ is the complete graph with $n$ vertices (i.e., every pair of vertices is connected), and the graph associated with the cube $B_\infty^n$ is the complement of the complete graph (i.e., no pair of vertices is connected). It is interesting to remark (\cite{Se}, \cite{Re3}) that standard Hanner polytopes are in one-to-one correspondence with the graphs which do not contain any induced path of edge length 3. Here, an induced path of a graph $G$ means a sequence of different vertices of $G$ such that each two adjacent vertices in the sequence are connected by an edge of $G$, and each two nonadjacent vertices in the sequence are not connected.

Given a graph $G$, a subset $J$ of the vertex set is called a \emph{clique} of $G$ if any two points in $J$ are connected by an edge of $G$. A subset $I$ of the vertex set is called an \emph{independent set} of $G$ if any two points in $I$ are not connected by an edge. 

Let $G$ be the graph associated with a standard Hanner polytope $H$ as above. It turns out \cite{Re3} that a point $v\in\R^n$ is a vertex of $H$ if and only if each coordinate of $v$ is $-1$, $0$, or $1$, and the set  $$\supp(v)=\set{j:\brA{v,e_j}\neq 0},$$ called the {\it support} of $v$, is a maximal independent set of $G$. Here, the maximality of cliques and independent sets comes from the partial order of inclusion. Similarly, a point $v^\star\in\R^n$ is a vertex of $H^\circ$ if and only if each coordinate of $v^\star$ is $-1$, $0$, or $1$, and the support of $v^\star$ is a maximal clique of $G$.
In addition, it is known \cite{Ha,HL,Re3} that every Hanner polytope $H$ satisfies {\it CL-property}: $\abs{\brA{v,v^\star}}=1$ for every vertex $v$ of $H$ and every vertex $v^\star$ of $H^\circ$. In other words, if $G$ is the graph associated with a standard Hanner polytope, then 
\begin{quote}
there exists a unique common element between any {\it maximal independent set} and any {\it maximal clique in $G$}.
\end{quote}
In this section we prove the following result:
\begin{theorem}\label{thm:delta_gaps}
Let $H$ be a standard Hanner polytope in $\R^n$ and $K_0$ a symmetric convex body in $\R^n$ with $\dBM(K_0,H)=1+\delta$ for small $\de>0$. Then, there exists a symmetric convex body $K$ of with $\dH(K,K_0)=1+o(\de)$  such that 
\begin{equation}\label{delta_gaps_pert}
|K||K^\circ|\ge V(X)V(X^*)+c(n)\de
\end{equation}
where $X=(x\sub{F})$, $X^*=(x\sub{F^*})$ are the sets of the $x\sub{F}$-points obtained from $K$, $K^\circ$ as in Definition \ref{def:AF}.
\end{theorem}

We start with some preparatory lemmas and propositions.
\begin{lemma}\label{lem:Ej}
For $j=1,\ldots,n$, consider $E_j=H\cap(e_j+e_j^\perp)$, which is a face of $H$ with centroid $e_j$. Then, 
\begin{align}
\aff(E_j)&=e_j+\spn\set{e_i\in H:i\nsim j}\label{eq:aff_Ej} \\
A\sub{E_j}&=e_j+\spn\set{e_i\in H:i\sim j},\label{eq:A_Ej}.
\end{align}
\end{lemma}
\begin{proof}
The equality \eqref{eq:aff_Ej} can be obtained from $$\conv\set{\pm e_i+e_j\in H:i\nsim j}\subset E_j\subset\aff\set{\pm e_i+e_j\in H:i\nsim j}.$$ Here, the first inclusion follows directly from definitions of $E_j$ and $i\nsim j$. We can get the second inclusion by showing that $E_j\subset e_i^\perp$ whenever $i\sim j$. Indeed, if $x\in E_j$, then the orthogonal projection of $x$ to $\spn\set{e_i,e_j}$ is $\brA{x,e_i}e_i+\brA{x,e_j}e_j\in H$. If $i\sim j$, then $\abs{\brA{x,e_i}}+\abs{\brA{x,e_j}}\le 1$ because the section of $H$ by $\spn\set{e_i,e_j}$ is the $2$-dimensional $\ell_1$-ball. Since $\brA{x,e_j}=1$, we get $\brA{x,e_i}=0$.

To prove \eqref{eq:A_Ej}, we use the induction on the dimension $n=\dim H$. If $n=1$, then \eqref{eq:A_Ej} is trivial. Now assume that \eqref{eq:A_Ej} is true for all standard Hanner polytopes of dimension less than $n$. Let $H\in\R^n$ be the $\ell_1$ or $\ell_\infty$ sum of two standard Hanner polytopes $H_1$ and $H_2$. Without loss of generality, we fix $j=1$ and assume $e_1\in \spn(H_1)$. Consider $F=H\cap(e_1+e_1^\perp)$ and $F_1=H_1\cap(e_1+e_1^\perp)$. From \eqref{eq:aff_Ej}, we get $\aff(F)=e_1+\spn\set{e_i\in H:i\nsim 1}$ and $\aff(F_1)=e_1+\spn\set{e_i\in H_1:i\nsim 1}$. So, if $H=H_1\oplus_1H_2$, then we get $\aff(F)=\aff(F_1)$, which implies $F=F_1\oplus_1\varnothing$ by Lemma \ref{lem:sum_faces}. If $H=H_1\oplus_\infty H_2$, then we get $\aff(F)=\aff(F_1)+\spn(H_2)$, which implies $F=F_1\oplus H_2$ by Lemma \ref{lem:sum_faces}. It follows from Definition \ref{def:AF} that
\begin{equation*}
A\sub{F}=
\begin{cases}
A\sub{F_1}+\spn(H_2), &\text{if } H=H_1\oplus_1H_2,\\
A\sub{F_1}, &\text{if } H=H_1\oplus_\infty H_2.
\end{cases}
\end{equation*}
Since $A\sub{F_1}=e_1+\spn\set{e_i\in H_1:i\sim1}$ by the induction hypothesis,  in any case we have $A\sub{F}=e_1+\spn\set{e_i\in H:i\sim1}$.
\end{proof}

\begin{proposition}\label{prop:step1}
Let $H$ be a standard Hanner polytope in $\R^n$, and $K_0$ be a symmetric convex body in $\R^n$ with $\dH(K_0,H)=\de$ for small $\de>0$. Consider the polytope $P_0$ defined by the convex hull of all the $x\sub{F}$-points of $K_0$. Then, there exists a polytope $K$ with $\dBM(K,P_0)=1+O(\de^2)$ such that $$\dH(K,H)=O(\de)\quad\text{and}\quad B_1^n\subset K\subset B_\infty^n.$$
\end{proposition}

\begin{proof}
Let $j\in\set{1,\ldots,n}$. Consider $E_j=H\cap(e_j+e_j^\perp)$, which is a face of $H$ with centroid $e_j$, and let $x_j=x\sub{E_j}$ be the $x\sub{F}$-point for $F=E_j$. Then, $\abs{x_j-e_j}=O(\de)$ by Proposition \ref{prop:construction}. Choose a hyperplane $\varphi_j+\varphi_j^\perp$ tangent to $K_0$ at $x_j$ and parallel to $A\sub{E_j}$.
First, we will prove
\begin{equation}\label{eq:hyperplane_delta}
\abs{\varphi_j-e_j}=O(\de).
\end{equation}
Indeed, if $i\sim j$, then $e_i$ is parallel to $A\sub{E_j}$ by \eqref{eq:A_Ej}, which implies
\begin{equation}\label{eq:ei_xij}
\brA{e_i,\varphi_j}=0\quad\text{whenever }i\sim j.
\end{equation}
If $i\nsim j$, then $e_i+e_j\in H$ gives $t^\pm(e_j\pm e_i)\in K_0$ for some $t^\pm=1+O(\de)$. So, the points $e_j\pm e_i$ are bounded by the hyperplanes $\pm(1+O(\de))(\varphi_j+\varphi_j^\perp)$. In addition, note that $\brA{x_j,\varphi_j}=\brA{\varphi_j,\varphi_j}$, $\abs{x_j-e_j}=O(\de)$, and $c_1\le\abs{\varphi_j}\le c_2$ for constants $c_1$, $c_2$ (by $\varphi_j/\abs{\varphi_j}^2\in\partial K_0^\circ$). Thus, if $i\nsim j$, then
\begin{align}
\brA{\pm e_i,\varphi_j}&=\brA{e_j\pm e_i,\varphi_j}-\brA{x_j,\varphi_j}+\brA{x_j-e_j,\varphi_j}\notag \\
&= \brA{e_j\pm e_i,\varphi_j}-\brA{\varphi_j,\varphi_j}+O(\de)\le O(\de). \label{eq:hyperplane_ij}
\end{align}
To estimate $\brA{e_j,\varphi_j}$, note that 
\begin{equation*}
\brA{\varphi_j,\varphi_j}=\sum_{i=1}^n\brA{e_i,\varphi_j}^2=\brA{e_j,\varphi_j}^2+\sum_{i\nsim j}\brA{e_i,\varphi_j}^2=\brA{e_j,\varphi_j}^2+O(\de^2)
\end{equation*}
is equal to $\brA{x_j,\varphi_j}=\brA{e_j,\varphi_j}+O(\de)$ because $x_j\in \varphi_j+\varphi_j^\perp$. It implies $\brA{e_j,\varphi_j}=1+O(\de)$, which, together with \eqref{eq:ei_xij} and \eqref{eq:hyperplane_ij}, completes the proof of \eqref{eq:hyperplane_delta}.

Choose $\te_j\in\Sp$ which is orthogonal to  
\begin{equation}\label{eq:theta_j}
\spn\Big(\!\set{e_j}\cup\set{e_i:i\nsim j}\!\Big)\cap\varphi_j^\perp+\spn\Set{x_i:i\sim j},
\end{equation}
and satisfies $\brA{e_j,\te_j}\ge0$.
Here, the dimension of the right hand side of \eqref{eq:theta_j} is at most $n-1$ because $\varphi_j\in\spn(\set{e_j}\cup\set{e_i:i\nsim j})$ by \eqref{eq:ei_xij}, so such $\te_j$ exists. Moreover, 
\begin{equation}\label{eq:theta_j_ej}
\abs{\te_j-e_j}=O(\de).
\end{equation}
Indeed, if $i\sim j$, then $\brA{x_i,\te_j}=0$, which implies $\brA{e_i,\te_j}=O(\de)$. Assume $i\nsim j$. Let $P_{\varphi_j^\perp}(e_i)$ be the orthogonal projection of $e_i$ onto $\varphi_j^\perp$. Then,
$P_{\varphi_j^\perp}(e_i)-e_i$ is parallel to $\varphi_j$, i.e., $P_{\varphi_j^\perp}(e_i)-e_i=\al\varphi_j$ where $\al=-\brA{e_i,\varphi_j}/\abs{\varphi_j}^2$ is obtained from $\langle P_{\varphi_j^\perp}(e_i),\varphi_j\rangle=0$. So, $|P_{\varphi_j^\perp}(e_i)-e_i|=\abs{\al\varphi_j}=\abs{\brA{e_i,\varphi_j}}/\abs{\varphi_j}=O(\de)$. Moreover, \eqref{eq:ei_xij} implies $P_{\varphi_j^\perp}(e_i)=e_i+\al\varphi_j\in\varphi_j^\perp\cap\spn(\set{e_j}\cup\set{e_k:k\nsim j})$. Thus, we have $\brA{P_{\varphi_j^\perp}(e_i),\te_j}=0$, which implies $\brA{e_i,\te_j}=O(\de)$. To estimate $\brA{e_j,\te_j}$, note that $|\te_j|^2=\sum\brA{e_i,\te_j}^2=\brA{e_j,\te_j}^2+O(\de^2)$. From $|\te_j|=1$ and $\brA{e_j,\te_j}>0$, we get $\brA{e_j,\te_j}=1+O(\de^2)$, which completes the proof of \eqref{eq:theta_j_ej}.

Consider the hyperplane $x_j+\te_j^\perp$. We will prove that $P_0=\conv\set{x\sub{F}}$ is bounded by hyperplanes $\pm(1+O(\de^2))(x_j+\te_j^\perp)$. Note that a face $F$ is contained in $E_j$ if and only if $c\sub{F}\in e_j+e_j^\perp$. Indeed, since $c\sub{F}\in{\rm int}(F)$, it can be written as $c\sub{F}=\sum_v\la_v v$ where $v$ runs over all vertices of $H$ contained in $F$, and all $\la_v$'s are positive numbers satisfying $\sum_v \la_v=1$. It implies that $\brA{c\sub{F},e_j}=1$ if and only if $\brA{v,e_j}=1$ for each vertex $v$ of $H$ in $F$, which is equivalent to $F\subset E_j$. Thus, the point $x\sub{F}$ is $O(\de)$-close to $x_j+\te_j^\perp$ if and only if $F\subset E_j$. It suffices to prove $\brA{x\sub{F},\te_j}\le \brA{x_j,\te_j}+O(\de^2)$ whenever $F\subset E_j$. For any face $F\subset E_j$,
\begin{align}
\brA{x\sub{F},\theta_j}-\brA{x_j,\theta_j} &=\brA{x\sub{F}-x_j,\varphi_j} + \brA{x\sub{F}-x_j,\theta_j-\varphi_j}\notag\\
&=\brA{x\sub{F}-x_j,\varphi_j} + \brA{sc\sub{F}-\varphi_j,\te_j-\varphi_j} + \brA{x\sub{F}-sc\sub{F},\te_j-\varphi_j}\label{eq:Aj_delta_sq}
\end{align}
where $s=\abs{\varphi_j}^2/\brA{c\sub{F},\varphi_j}=1+O(\de)$.
The first term in \eqref{eq:Aj_delta_sq} satisfies $\brA{x\sub{F}-x_j,\varphi_j}\le0$ because $\varphi_j+\varphi_j^\perp$ is tangent to $K_0$ at $x_j$. For the second term in \eqref{eq:Aj_delta_sq}, note that $sc\sub{F}-\varphi_j \in \varphi_j^\perp$ by choice of $s$. Moreover, since $\brA{e_i,\varphi_j}=0=\brA{e_i,c\sub{F}}$ for each $i\sim j$ by \eqref{eq:ei_xij} and \eqref{eq:aff_Ej}, the point $sc\sub{F}-\varphi_j$ belongs to $\varphi_j^\perp\cap\spn(\set{e_j}\cup\set{e_i:i\nsim j})$, which gives $sc\sub{F}-\varphi_j\in\te_j^\perp$. Thus, $\brA{sc\sub{F}-\varphi_j,\te_j-\varphi_j}=0$. The last term in \eqref{eq:Aj_delta_sq} satisfies  $\brA{x\sub{F}-sc\sub{F},\te_j-\varphi_j}\le \abs{x\sub{F}-sc\sub{F}}\cdot\abs{\te_j-\varphi_j}=O(\de^2)$. Finally, $\brA{x\sub{F},\te_j}\le \brA{x_j,\te_j}+O(\de^2)$.

Take a linear transformation $T_1$ which maps $x_j+\te_j^\perp$ to $e_j+e_j^\perp$ for each $j=1,\ldots,n$. Indeed, to see the existence of such a linear transformation, consider the parallelepiped $Q$ bounded by the hyperplanes $\pm x_j+\te_j^\perp$ for $j=1,\ldots,n$. If $T_1$ is the linear transformation which maps the centroid of the facet $Q\cap(x_j+\te_j^\perp)$ of $Q$ to $e_j$ for $j=1,\ldots,n$, then it is a desired linear transformation. Here, both $x_j$ and $\te_j$ are $O(\de)$-close to $e_j$, so is the centroid of $Q\cap(x_j+\te_j^\perp)$. Thus, $\norm{T_1-Id}=O(\de)$ where $Id$ is the identity operator on $\R^n$. Let $P_1=T_1P_0$, and $x_j'=T_1x_j$ for each $j$. Then, $\dH(P_1,P_0)=O(\de)$. Also, $\dH(P_1,P_1\cap B_\infty^n)=O(\de^2)$ because $P_0$ is bounded by hyperplanes $\pm(1+O(\de^2))(x_j+\te_j^\perp)$ for each $j$. Moreover, if $i\sim j$, then $x_i$ is parallel to $x_j+\te_j^\perp$ by \eqref{eq:theta_j}. So, $x_i'$ is parallel to $e_j+e_j^\perp$ whenever $i\sim j$, that is,
\begin{equation}\label{eq_perturb}
\brA{x_i',e_j}=0\quad\text{whenever $i\sim j$}.
\end{equation}
Let $(1-c\de)H\subset P_1\subset (1+c\de)H$ for some constant $c>0$. By \eqref{eq_perturb} we can write $$x_1'=e_1+\sum_{j\in J}\la_je_j$$ where $J=\set{j: j\nsim1}$ and $\la_j=O(\de)$. Consider the point $$z=\frac{1-c\de}{\la}\sum_{j\in J}\abs{\la_j}\big(e_1-{\rm sign}(\la_j)\,e_j\big),$$ where $\la=\sum|\la_j|=O(\de)$. Since $e_1\pm e_j\in H$ for each $j\in J$, we get $\frac1{\la}\sum_{j\in J}\abs{\la_j}(e_1-{\rm sign}(\la_j)e_j)\in H$. So, $z\in P_1$. Moreover, consider the point 
\begin{equation*}
\frac{x_1'+\frac{\la}{1-c\de}z}{1+\frac{\la}{1-c\de}}=\frac{1+\la}{1+\frac{\la}{1-c\de}}e_1=:r_1e_1,
\end{equation*}
where $r_1=1-c\de\la+O(\de^2)=1+O(\de^2)$. Then, $r_1e_1\in P_1$ because it is a convex combination of $x_1'\in P_1$ and $z\in P_1$. Repeat the above procedure to get $r_2e_2,\ldots,r_ne_n\in P_1$ for $j=2,\ldots,n$ where $r_j=1-O(\de^2)$ for each $j$. 

Take a linear transformation $T_2$ which maps $r_je_j$ to $e_j$ for each $j$. Then $\norm{T_2-Id}=O(\de^2)$ because $r_j=1-O(\de^2)$. Let $T=T_2T_1$ and $K=TP_0\cap B_\infty^n$. Then, $\dBM(K,P_0)=1+O(\de^2)$, $\dH(K,H)=O(\de)$ and $B_1^n\subset K\subset B_\infty^n$.
\end{proof}

For each vertex $v$ of a standard Hanner polytope in $\R^n$, consider the subspace $\R^{(v)}$, the {\it lower dimensional cube} $B_\infty^{(v)}$, and the {\it lower dimensional cross-polytope} $B_1^{(v)}$, defined by $$\R^{(v)}=\spn\set{e_j:j\in\supp(v)},\qquad B_\infty^{(v)}=B_\infty^n\cap\R^{(v)},\qquad B_1^{(v)}=B_1^n\cap\R^{(v)}.$$  Then, we can see that
\begin{align*}
B_\infty^{(v)}\,&=H\cap\R^{(v)}\,=H|\R^{(v)},\quad\text{for every vertex $v$ of $H$,}\\
B_1^{(v^{\!\star}\!)}&=H\cap\R^{(v^{\!\star}\!)}=H|\R^{(v^{\!\star}\!)}\!, \quad\text{for every vertex $v^\star$ of $H^\circ$,}
\end{align*}
where $H|\R^{(v)}$ denotes the orthogonal projection of $H$ to $\R^{(v)}$.
\begin{proposition}\label{prop:step2}
Let $H$ be a standard Hanner polytope in $\R^n$ and $K$ a symmetric convex body in $\R^n$ with $\dH(K,H)=\de$ for small $\de>0$. Then, there exists either a vertex $v$ of $H$ with $\dH\brR{K|\R^{(v)},B_\infty^{(v)}}\ge c\de$, or a vertex $v^\star$ of $H^\circ$ with $\dH\brR{K^\circ|\R^{(v^{\!\star}\!)},B_\infty^{(v^{\!\star}\!)}}\ge c\de$, where $c=c(n)$ is a positive constant.
\end{proposition}

For the proof by contradiction, we assume that
\begin{align}
\dH\brR{K\,|\,\R^{(v)},B_\infty^{(v)}}&=o(\de)\quad\text{for every vertex $v$ of $H$},\label{dist_projection}\\
\dH\brR{K\cap\R^{(v^{\!\star}\!)},B_1^{(v^{\!\star}\!)}}&=o(\de)\quad\text{for every vertex $v^\star$ of $H^\circ$.} \label{dist_section}
\end{align}
Here, the equality in \eqref{dist_section} is equivalent to $\dH\brR{K^\circ|\R^{(v^{\!\star}\!)},B_\infty^{(v^{\!\star}\!)}}=o(\de)$ because the polar of $K\cap\R^{(v^{\!\star}\!)}$ in $\R^{(v^{\!\star}\!)}$ is the same as $K^\circ|\R^{(v^{\!\star}\!)}$.
We need the following two lemmas to prove Proposition \ref{prop:step2}.
\begin{lemma}\label{lem:AF_vertex}
Let $v$ be a vertex of a standard Hanner polytope $H$ in $\R^n$, and $A_v$ the affine subspace for the (zero-dimensional) face $v$ defined in Definition \ref{def:AF}. Then,
$$A_v\cap\R^{(v)}=A_v |\,\R^{(v)}.$$  
\end{lemma}
\begin{proof}
We use the induction on $n=\dim H$. The statement is trivial if $n=1$. Assume that Lemma \ref{lem:AF_vertex} is true for all standard Hanner polytopes of dimension less than $n$. Let $H\subset\R^n=\R^{n_1}\oplus\R^{n_2}$ be the $\ell_1$ or $\ell_\infty$ sum of standard Hanner polytopes $H_1\subset\R^{n_1}$, $H_2\subset\R^{n_2}$. Let $v$ be a vertex of $H$. First, consider the case $H=H_1\oplus_1 H_2$. Then, it follows from Remark \ref{rmk:face_form} that $v=v_1\oplus_1\varnothing$ or $\varnothing\oplus_1v_2$ where $v_1$, $v_2$ are vertices of $H_1$, $H_2$ respectively. Say, if $v=v_1\oplus_1\varnothing$, then $\R^{(v)}=\R^{(v_1)}$, and Definition \ref{def:AF} gives $A_v=A_{v_1}+\spn(H_2)$. So, $A_v\cap\R^{(v)}=(A_{v_1}+\spn(H_2))\cap\R^{(v_1)}=A_{v_1}\cap\R^{(v_1)}$ is equal to $A_{v_1} |\,\R^{(v_1)}=A_v |\,\R^{(v)}$ by the induction hypothesis. Now, consider the case $H=H_1\oplus_\infty H_2$. In this case, $v$ can be written as $v=v_1\oplus v_2$ where $v_1$, $v_2$ are vertices of $H_1$, $H_2$ respectively. So,  $\R^{(v)}=\R^{(v_1)}+\R^{(v_2)}$ and $A_v=A_{v_1}+A_{v_2}+\spn\set{v_1/n_1-v_2/n_2}$. Since $v_1/n_1-v_2/n_2\in\R^{(v)}$, we get 
\begin{align*}
A_v\cap\R^{(v)} &\supset A_{v_1}\cap\R^{(v)}+A_{v_2}\cap\R^{(v)}+\spn\set{v_1/n_1-v_2/n_2}\\
&=A_{v_1}\cap\R^{(v_1)}+A_{v_2}\cap\R^{(v_2)}+\spn\set{v_1/n_1-v_2/n_2},
\end{align*}
and 
\begin{align*}
A_v|\,\R^{(v)}&\subset A_{v_1} |\,\R^{(v)} + A_{v_2} |\,\R^{(v)}+\spn\set{v_1/n_1-v_2/n_2}\\
&= A_{v_1} |\,\R^{(v_1)} + A_{v_2} |\,\R^{(v_2)}+\spn\set{v_1/n_1-v_2/n_2}
\end{align*}
Since $A_{v_1}\cap\R^{(v_1)}=A_{v_1} |\,\R^{(v_1)}$ and $A_{v_2}\cap\R^{(v_2)}=A_{v_2} |\,\R^{(v_2)}$ by the induction hypothesis, we have $A_v\cap\R^{(v)}\supset A_v |\,\R^{(v)}$. The opposite inclusion is trivial.
\end{proof}

Under the assumption \eqref{dist_projection}, Lemma \ref{lem:AF_vertex} implies that the orthogonal projection of the point $x_v$ (the $x\sub{F}$-point for $F=v$) to $\R^{(v)}$ is $o(\de)$-close to $v$. In other words, for any vertex $v$, the assumption \eqref{dist_projection} gives
\begin{equation}\label{vertex_projection}
\brA{x_v,e_j}=\brA{v,e_j}+o(\de)\qquad\text{for every $j\in\supp(v)$}.
\end{equation}
To prove \eqref{vertex_projection}, let $\tilde{x}_v$, $\tilde{A}_v$, $\tilde{K}$ be the orthogonal projections to $\R^{(v)}$ of $x_v$, $A_v$, $K$ respectively. Then $\tilde{A}_v=A_v\cap\R^{(v)}$ by Lemma \ref{lem:AF_vertex}. It gives that $t\tilde{A}_v$ is tangent to $\tilde{K}$ at $\tilde{x}_v$ whenever $tA_v$ is tangent to $K$ at $x_v$. Moreover, $\tilde{A}_v$ satisfies the conditions (a), (b) for $P=B_\infty^{(v)}$ given in Section \ref{general_construction}. By Proposition \ref{prop:construction}, the assumption \eqref{dist_projection} implies $|\tilde{x}_v-v|=o(\de)$, which complete the proof of \eqref{vertex_projection}.

The next lemma gives an estimate for the $j$-th coordinates of the $x_v$-points when $j\not\in\supp(v)$.

\begin{lemma}\label{lem:step2}
Let $H$ be a standard Hanner polytope in $\R^n$ and $K$ a symmetric convex body in $\R^n$ with $(1-\de)H\subset K\subset(1+\de)H$. Let $v$, $w$ be vertices of $H$ with $\supp(v)=\supp(w)$. Suppose that $1\notin\supp(v)$ and
\begin{equation}
\brA{v,e_i}=\brA{w,e_i}\quad\forall i\nsim1.
\end{equation}
Then, under the assumptions \eqref{dist_projection} and \eqref{dist_section}, we have $\brA{x_v,e_1}=\brA{x_w,e_1}+o(\de)$.
\end{lemma}

\begin{proof}
It suffices to consider the case that two vertices $v$, $w$ have different values in only one coordinate and the same values in the other coordinates; for the general case, consider a sequence of vertices $v=v_1, v_2,\ldots,v_k=w$ such that any two consecutive vertices $v_j$, $v_{j+1}$ have different values in only one coordinate. Thus, we may assume that  $\brA{v,e_j}=\brA{w,e_j}$ for any $j=1,\ldots,n-1$, and $\brA{v,e_n}\neq\brA{w,e_n}$.
Let $I=\supp(v)$. Then $v$ and $w$ can be written as $$v=e_n+\sum_{j\in I\setminus\set{n}}e_j \qquad\text{and}\qquad w=-e_n+\sum_{j\in I\setminus\set{n}}e_j.$$ Under the assumption \eqref{dist_projection}, as in \eqref{vertex_projection},  the points $x_v$, $x_w\in K$ can be expressed by 
\begin{align*}
x_v&=e_n+\sum_{j\in I\setminus\set{n}}e_j + \sum_{j\notin I}a_je_j + o(\de)\\
x_w&=-e_n+\sum_{j\in I\setminus\set{n}}e_j + \sum_{j\notin I}b_je_j + o(\de)
\end{align*}
where $|a_j|=O(\de)$ and $|b_j|=O(\de)$ follow from $|x_v-v|=O(\de)$. So, 
\begin{equation}\label{eq:midpoints_xv}
\frac{x_v-x_w}{2} = e_n + \sum_{j\notin I}\xi_j e_j + o(\de),\quad\xi_j=\frac{a_j-b_j}{2}=O(\de). 
\end{equation}
Here, we may assume that each $\xi_j$ is nonnegative; if $\xi_j<0$, then replace $e_j$ with $-e_j$.  Then we claim that 
\begin{equation}\label{eq:xij_odelta}
\xi_1=o(\de).
\end{equation}
To prove \eqref{eq:xij_odelta}, let $\bar{x}=e_n + \sum_{j\notin I}\xi_j e_j$. Then, $(1-\eps)\bar{x}\in K$ for some $\eps=o(\de)$ by \eqref{eq:midpoints_xv}. Take a maximal clique $J$ containing $\set{1,n}$ in the graph $G$ associated with $H$, and let  $v^\star=\sum_{j\in J}e_j$, which is a vertex of $H^\circ$ with support $J$.
First, if $I\cup J=\set{1,\ldots,n}$, then $\bar{x}=e_n + \sum_{j\notin I}\xi_j e_j\in\R^{(v^{\!\star}\!)}$. By \eqref{dist_section}, the $\ell_1$-norm of $(1-\eps)\bar{x}$ satisfies 
$\norm{(1-\eps)\bar{x}}_1\le 1+o(\de)$, which gives $\xi_1=o(\de)$. Thus, it remains to consider the case that $M:=\set{1,\ldots,n}\setminus(I\cup J)$ is not empty. It follows from maximality of $J$ that for every $k\in M$, there exists $j_k\in J$ with $j_k\nsim k$. Consider the point 
\begin{equation*}
y=(1-\de)\sum_{k\in M}\brR{\la\,\frac{e_{j_k}+e_n}{2}+\la_k(e_{j_k}-e_k)},
\end{equation*}
where $\la$ and all $\la_k$'s are non-negative and obtained from $\sum_{k\in M}(\la+\la_k)=1$ and $\la_k=\xi_k/(\sqrt{\de}+\de)$.
Then, $y\in K$ because it is a convex combination of points $(1-\de)\frac{e_{j_k}+e_n}{2}\in K$ and $(1-\de)(e_{j_k}-e_k)\in K$ for $k\in M$. Let $z=\sqrt{\de}y + (1-\sqrt{\de})\bar{x}$.
That is, 
\begin{align*}
z&=\sqrt{\de}(1-\de)\sum_{k\in M}\brR{\la\,\frac{e_{j_k}+e_n}{2}+\la_k(e_{j_k}-e_k)} + (1-\sqrt{\de})\Big(e_n + \sum_{j\notin I} \xi_j e_j\Big) \\
&=\brR{\sqrt{\de}(1-\de)|M|\la/2+(1-\sqrt{\de})}e_n + \sqrt{\de}(1-\de)\sum_{k\in M}(\la/2+\la_k) e_{j_k} + (\star)
\end{align*}
where 
\begin{align*}
(\star)&=-\sqrt{\de}(1-\de)\sum_{k\in M}\la_ke_k + (1-\sqrt{\de})\sum_{j\notin I} \xi_j e_j\\
&= -\sqrt{\de}(1-\de)\sum_{k\in M}\brR{\la_k-\frac{1\!-\!\sqrt{\de}}{\sqrt{\de}(1\!-\!\de)}\,\xi_k}e_k + (1\!-\!\sqrt{\de})\!\!\!\sum_{j\in J\setminus\set{n}}\!\xi_j e_j  =(1\!-\!\sqrt{\de})\!\!\!\sum_{j\in J\setminus\set{n}}\!\xi_j e_j.
\end{align*}
Thus, 
\begin{equation}\label{eq:z_expression}
z=\brR{\sqrt{\de}(1-\de)|M|\la/2+1\!-\!\sqrt{\de}}e_n + \sqrt{\de}(1-\de)\sum_{k\in M}(\la/2+\la_k) e_{j_k} + (1\!-\!\sqrt{\de})\!\!\sum_{j\in J\setminus\set{n}}\!\!\xi_j e_j.
\end{equation}
Note that $(1-\eps)z\in K\cap\R^{(v^{\!\star}\!)}$ because $(1-\eps)y,(1-\eps)\bar{x}\in K$ and $z\in\R^{(v^{\!\star}\!)}$. So, \eqref{dist_section} gives $\norm{(1-\eps)z}_1\le1+o(\de)$, that is, $\norm{z}_1\le1+o(\de)$. Since all coefficients of the terms in \eqref{eq:z_expression} are non-negative by the assumption $\xi_j\ge0$, the $\ell_1$-norm of $z$ is
\begin{align*}
\norm{z}_1 &= \sqrt{\de}(1-\de)|M|\la/2+1\!-\!\sqrt{\de} + \sqrt{\de}(1-\de)\sum_{k\in M}(\la/2+\la_k) + (1\!-\!\sqrt{\de})\!\!\sum_{j\in J\setminus\set{n}}\!\!\xi_j \\
&=1-\sqrt{\de} + \sqrt{\de}(1-\de)\sum_{k\in M}(\la+\la_k) + (1\!-\!\sqrt{\de})\!\!\sum_{j\in J\setminus\set{n}}\!\!\xi_j = 1-\de\sqrt{\de} + (1\!-\!\sqrt{\de})\!\!\sum_{j\in J\setminus\set{n}}\!\!\xi_j 
\end{align*}
Therefore, $\norm{z}_1\le  1+o(\de)$ implies
\begin{equation*}
\sum_{j\in J\setminus\set{n}}\xi_j \le \frac{\de\sqrt{\de}}{1-\sqrt{\de}} +o(\de) = O(\de\sqrt{\de})  +o(\de) = o(\de),
\end{equation*}
which completes the proof.
\end{proof}

\begin{proof}[Proof of Proposition \ref{prop:step2}]
Suppose that $K$ and $H$ satisfy \eqref{dist_projection}, \eqref{dist_section}. It follows from $\dH(K,H)=\de$ that there exist a constant $c=c(n)$ and a vertex $v$ of $H$ or $H^\circ$ with $|x_v-v|\ge c\de$. Otherwise, $|x_v-v|=o(\de)$ for every vertex $v$ of $H$ or $H^\circ$, which gives $\dH(K,H)=o(\de)$; contradiction. In addition, since $\brA{x_v,e_j}=\brA{v,e_j}+o(\de)$ for every $j\in\supp(v)$ by \eqref{vertex_projection}, there exists $j\not\in\supp(v)$ such that $\abs{\brA{x_v,e_j}}\ge c\de$. Without loss of generality, we may assume that $v$ is a vertex of $H$,
\begin{equation}
1\notin\supp(v),\quad\text{and}\quad\brA{x_v,e_1}\ge c\de.
\end{equation}
Let $I=\supp(v)$. Take a maximal clique $J$ containing $1$ in the graph $G$ associated with $H$, and let $m$ be the unique common point between $I$ and $J$. Without loss of generality, we write $J=\set{1,2,\cdots,m}$, $I\cap J=\set{m}$, and $v=\sum_{i\in I}e_i$. For $k=1,\cdots,m$, define the subset $I_k$ of $I$ by 
\begin{align*}
I_1&=\Set{i\in I: i\nsim1}\\
I_k&=\Set{i\in I\backslash(I_1\cup\cdots\cup I_{k-1}): i\nsim k},\quad\text{if }k=2,\ldots,m.
\end{align*}
The maximality of $J$ gives that for each $i\in I\setminus\set{m}$ there exists $j\in J$ that is not connected to $i$, so
\begin{equation}\label{disjoint_union_I}
I=(I_1\cup\cdots\cup I_m)\cup\set{m}.
\end{equation}
\begin{center}
\includegraphics[scale=.7,page=1]{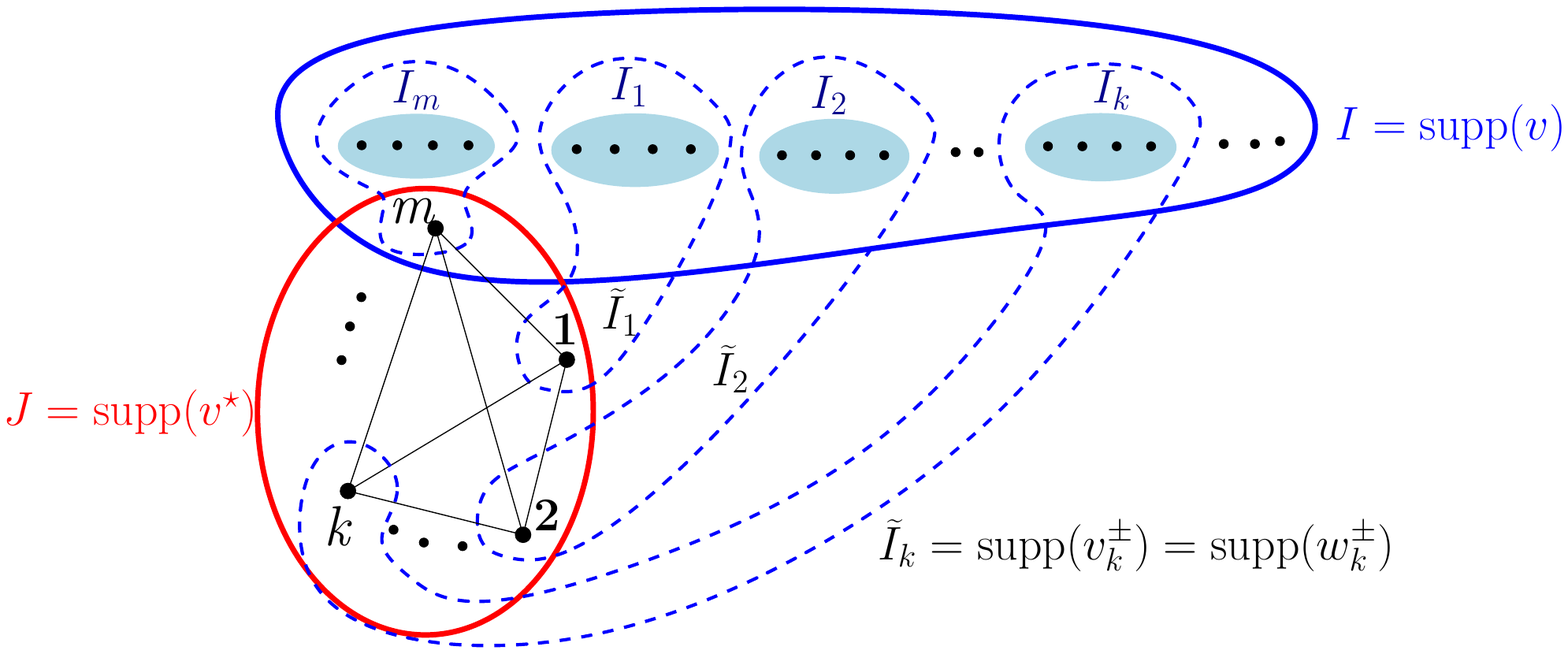}
\end{center}
Since each $I_k\cup\set{k}$ is an independent set, we can choose a maximal independent set $\widetilde{I}_k$ containing $I_k\cup\set{k}$ for $k=1,\cdots,m-1$. Consider the vertices $v_k^+$, $v_k^-$ with support $\widetilde{I}_k$ for $k=1\ldots,m-1$ defined by $$v_k^\pm=e_k-\sum_{I_k}e_i \pm\!\!\sum_{\widetilde{I}_k\setminus(I_k\cup\set{k})}\!\!e_i,$$
and the vertices $v_m^+$, $v_m^-$ with support $I$ defined by $$v_m^\pm =e_m+\!\!\!\!\sum_{I_1\cup\cdots\cup I_{m\!-\!1}}\!\!\!\!\!\!\!e_i \,\,\pm\,\sum_{I_m}\,e_i.$$
For $j=1,\ldots,m$, let 
\begin{equation}\label{eq:def_tj}
t_j = {\rm sign}\brR{\sum_{k\in J\setminus\set{j}}\brA{x_{v_k^+}\!+x_{v_k^-},\,\,e_j}}.
\end{equation}
Let $v^\star$ be the vertex of $H^\circ$ with support $J$ defined as $$v^\star=\sum_{j=1}^m t_je_j.$$ 
Consider the vertices $w_1^+,w_1^-,\ldots,w_m^+,w_m^-$ defined by
\begin{align*}
w_k^\pm&=t_ke_k-\sum_{I_k}e_i \,\,\pm\!\!\!\sum_{\widetilde{I}_k\setminus(I_k\cup\set{k})}\!\!\!\!e_i \quad\text{for }k=1,\ldots,m\!-\!1 \\
w_m^\pm&=t_me_m+\!\!\!\!\sum_{I_1\cup\cdots\cup I_{m\!-\!1}}\!\!\!\!\!\!\!e_i \,\,\pm\,\sum_{I_m}\,e_i.
\end{align*}
Since $j\in\supp(w_j^\pm)$ for each $j$, \eqref{vertex_projection} implies
\begin{equation}\label{eq:wj_ej}
\brA{x_{w_j^\pm},e_j}=\brA{w_j^\pm,e_j}+o(\de)=t_j+o(\de)\quad\text{for $j=1,\ldots,m$}.
\end{equation}
Note that $w_k^\pm$, $v_k^\pm$ have the same support and may have different values only in the $k$-the coordinate. If $k\in J$ is different from $j$, then $j\sim k$ and by Lemma \ref{lem:step2} we get
\begin{equation*}
\brA{x_{w_k^\pm},e_j}=\brA{x_{v_k^\pm},e_j}+o(\de) \quad\text{for each }k\in J\setminus\set{j}.
\end{equation*}
Together with \eqref{eq:def_tj}, it implies
\begin{equation}\label{eq:t_sign}
\sum_{k\in J\setminus\set{j}}\brA{x_{w_k^+}+x_{w_k^-},\,t_je_j}=t_j\sum_{k\in J\setminus\set{j}}\brA{x_{v_k^+}+x_{v_k^-},\,\,e_j}+o(\de) \ge o(\de). 
\end{equation}
From the construction of $w_k^\pm$'s and $v^\star$, we can see that
\begin{equation}\label{relation_with_dual}
\sum_{k=1}^m \frac{w_k^+ + w_k^-}{2} = v^\star.
\end{equation}
Let $F$ be the dual face of $v^\star$, i.e., $F=(v^\star)^*=\set{x\in H:\brA{x,v^\star}=1}$. Consider the point $$y=\frac{1}{m}\sum_{k=1}^m \frac{x_{w_k^+}+x_{w_k^-}}{2},$$ which belongs to $K$. To get a lower bound of $\brA{y,v^\star}$, let us split it into three parts as follows:
\begin{align*}
\brA{y,v^\star} &= \brA{\frac1m\sum_{k=1}^m \frac{x_{w_k^+}+x_{w_k^-}}{2},\sum_{j=1}^m t_je_j}= \frac1m\sum_{j=1}^m \sum_{k=1}^m \brA{\frac{x_{w_k^+}+x_{w_k^-}}{2},t_je_j} \\
&= \quad{\rm(I)}\quad+\quad{\rm(II)}\quad+\quad{\rm(III)}
\end{align*}
where 
\begin{align*}
{\rm(I)} &= \frac1m\sum_{j=1}^m \brA{\frac{x_{w_j^+}+x_{w_j^-}}{2},t_je_j}, \\
{\rm(II)} &= \frac{t_1}{m}\sum_{k\in J\setminus\set{1}} \brA{\frac{x_{w_k^+}+x_{w_k^-}}{2},e_1},\\
{\rm(III)} &= \frac1m\sum_{j=2}^m t_j\!\!\!\sum_{k\in J\setminus\set{j}} \brA{\frac{x_{w_k^+}+x_{w_k^-}}{2},e_j}.
\end{align*}
The first term satisfies ${\rm(I)}=\frac1m\sum_{j=1}^mt_j^2+o(\de)=1+o(\de)$ by \eqref{eq:wj_ej}, and the third term ${\rm(III)}$ is non-negative or $o(\de)$-small by \eqref{eq:t_sign}. For ${\rm(II)}$, note that $1\not\in\supp(w_k^\pm)$ for $k=2,\ldots,m$ because $k\sim1$, and every element of $I\setminus I_1$ is connected to $1$ due to the choice of $I_1$. The vertices $w_m^\pm$ and $v$ have the same support $I$ and may have different values only in the coordinates $i\in I_m\cup\set{m}\subset I\setminus I_1$. Thus, Lemma \ref{lem:step2} gives
\begin{equation}\label{eq:e1_xwm}
\brA{x_{w_m^\pm},e_1}=\brA{x_v,e_1}+o(\de)
\end{equation}
For $k=2,\ldots,m-1$, the vertices $w_k^+$ and $-w_k^-$ have the same support $\widetilde{I}_k$ and different values only in the coordinates $i\in I_k\cup\set{k}\subset I\setminus I_1$. By Lemma \ref{lem:step2}, we get
\begin{equation}\label{eq:e1_xwk}
\brA{x_{w_k^+},e_1}=\brA{x_{-w_k^-},e_1}+o(\de)\quad\text{for }k=2,\ldots,m-1,
\end{equation}
which gives $\brA{x_{w_k^+}+x_{w_k^-},e_1}=o(\de)$ by symmetry of $K$.
Thus, by \eqref{eq:e1_xwm} and  \eqref{eq:e1_xwk}, $$\sum_{k=2}^m \brA{\frac{x_{w_k^+}+x_{w_k^-}}{2},e_1}=\brA{x_v,e_1}+o(\de)\ge c\de+o(\de).$$ It gives $t_1=1$ and hence the second term ${\rm(II)}$ is at least $\frac{c}{m}\de+o(\de)$. Finally we have 
\begin{equation}\label{delta_above}
\brA{y,v^\star}\ge 1+\frac{c}{m}\,\de+o(\de).
\end{equation}
On the other hand, $y$ is a point of $K$ which is $O(\de)$-close to the point $\frac1m v^\star$ on the facet $F$ by \eqref{relation_with_dual}. In fact, $\frac1m v^\star$ is an interior point of  $F$. Indeed, 
\begin{align*}
\frac1m v^\star=\frac1{\abs{J}}\sum_{j\in J}t_je_j =\frac1{\abs{J}}\sum_{j\in J}\Big(\frac1{\abs{V_j}}\sum_{v\in V_j}\frac{v+\bar{v}}2\Big) =\sum_{j\in J}\sum_{v\in V_j}\frac1{\abs{J}\abs{V_j}}\,v,
\end{align*}
where $V_j=\set{v\in\ext(F):j\in\supp(v)}$ for $j\in J$ and $\bar{v}=2t_je_j-v\in V_j$ for $v\in V_j$. So, $\frac1m v^\star$ can be written as $\frac1m v^\star=\sum_{v\in\ext(F)}\la_v v$ where all $\la_v$'s are positive and satisfy $\sum\la_v=1$. It implies $\frac1m v^\star$ is an interior point of $F$. In addition, under the assumption \eqref{dist_section}, the point $\frac1m v^*$ is $o(\de)$-almost on the boundary of $K$ because $\frac1m v^\star\in \partial B_1^{(v^{\!\star}\!)}$. This contradicts \eqref{delta_above}. Indeed, write $y=y_0+\frac{t}{m}v^\star$ for $y_0\in F$ and $t>0$, and take a point $z$ on the boundary of $F$ meeting with the ray from $y_0$ to $\frac1m v^\star$. Then, $(1-c'\de)z\in K$ for some $c'>0$, and $\frac1m v^\star=(1-\mu)y_0+\mu z$ for some $\mu=O(\de)$ because $\frac1m v^\star$ is an interior point of the facet $F$. Then, the point 
\begin{equation}\label{eq:convex_combin_y_z}
\frac{(1-\mu)y +\frac{\mu}{1-c'\de}(1-c'\de)z}{1-\mu+\frac{\mu}{1-c'\de}}=\frac{1+(1-\mu)t}{1-\mu+\frac{\mu}{1-c'\de}}\cdot\frac1m v^\star = \big(1+(1-\mu)t+O(\de^2)\big)\cdot\frac1m v^\star
\end{equation}
belongs to $K\cap\R^{(v^{\!\star}\!)}$ because it is a convex combination of two points $y$ and $(1-\de)z$ in $K$. So, the $\ell_1$-norm of \eqref{eq:convex_combin_y_z} is at most $1+o(\de)$ by the assumption \eqref{dist_section}, which gives $t=o(\de)$. It implies $\brA{y,v^\star}=1+t=1+o(\de)$ which contradicts \eqref{delta_above} and completes the proof.
\end{proof}

\begin{proposition}\label{prop:step3}
Let $H$ be a standard Hanner polytope in $\R^n$ and $K$ a symmetric convex body satisfying $B_1^n\subset K\subset B_\infty^n$. If $\dH(K|\R^{(v)},B_\infty^{(v)})=\de$ for small $\de>0$ and some vertex $v$ of $H$, then $$|K||K^\circ|\ge V(X)V(X^*)+c(n)\de,$$ where $X=(x\sub{F})$, $X^*=(x\sub{F^*})$ are the $x\sub{F}$-points obtained from $K$, $K^\circ$ as in Definition \ref{def:AF}.
\end{proposition}

\begin{proof}
Without loss of generality, we may assume that $(1-\de)B_\infty^{(v)} \subset K|\R^{(v)}\subset B_\infty^{(v)}$ and $(1-\de)v$ is on the boundary of $K|\R^{(v)}$. In addition, taking an appropriate coordinate system, we may assume that $\R^{(v)}=\spn\set{e_1,\ldots,e_m}=\R^m$ and $\brA{x_v,e_1}=\min_{1\le j\le m} |\brA{x_v,e_j}|$. Then, $$\brA{x_v,e_1}\le 1-\de.$$ Indeed, if $\brA{x_v,e_j}>1-\de$ for all $j=1,\ldots,m$, then the orthogonal projection of $x_v$ to $\R^m$ is in the interior of $v+\de B_\infty^m$ by Lemma \ref{lem:AF_vertex}, but it is impossible because $K|\R^m\supset (1-\de)B_\infty^m$ and $(1-\de)v$ is on the boundary of $K|\R^m$.

Let $\eps=1-\max\Set{\brA{c\sub{F},c\sub{G^*}}:F\not\subset G,\text{ faces of } H}$. Then $\eps>0$. Indeed, since $c\sub{F}\in{\rm int}(F)$, the centroid can be written as $c\sub{F}=\sum_v\la_v v$ where $v$ runs over all vertices of $H$ contained in $F$, and all $\la_v$'s are positive numbers satisfying $\sum_v \la_v=1$. So, if $\brA{c\sub{F},c\sub{G^*}}=1$ then we have $\brA{v,w}=1$ for any vertex $v$ in $F$ and any vertex $w$ in $G^*$. It implies $G^*\subset F^*$; hence $F\subset G$. Thus, $\max\Set{\brA{c\sub{F},c\sub{G^*}}:F\not\subset G}<1$.

Let $E_1=\set{y\in H:\brA{y,e_1}=1}$ be the face of $H$ with centroid $e_1$. Then, $v\subset E_1$, and $x\sub{E_1^*}=e_1$ because $K\subset B_\infty^n$ and $e_1\in K$.  Consider the point 
\begin{align*}
x^\star&=(1-\de)x\sub{E_1^*}+(1+\eps)\de x\sub{v^*}\\
&=(1-\de)e_1+(1+\eps)\de x\sub{v^*},
\end{align*}
where $v^*$ is the dual face of (a zero-dimensional face) $v$.
Then 
\begin{align*}
\brA{x_v,x^\star} &= (1-\de)\brA{x_v,e_1} +(1+\eps)\de\\
 &\le  (1-\de)^2 +(1+\eps)\de =1-(1-\eps)\de +\de^2
\end{align*}
and, for any face $F$ of $H$ different from $v$,
\begin{align*}
\brA{x\sub{F},x^\star} &= (1-\de)\brA{x\sub{F},e_1} +(1+\eps)\de\brA{x\sub{F},x\sub{v^*}}\\
&\le  (1-\de) +(1+\eps)\de\brS{\brA{c\sub{F},c\sub{v^*}}+O(\de)}\\
&\le  (1-\de) +(1+\eps)\de(1-\eps)+O(\de^2) = 1-\eps^2\de +O(\de^2).
\end{align*}
Therefore, for small $\de>0$, $$x^\star\in\brS{(1+\eps^2\de)\brR{\bigcup\nolimits_\F X_\F}}^\circ$$ where $X_\F$ is the simplex defined from $X=(x\sub{F})$ as in \eqref{def_simplex}. 

First, consider the case that $K\subset(1+\eps^2\de)\brR{\bigcup_\F X_\F}$. Then $x^\star\in K^\circ$. In addition, the point $x^\star$ is outside of the polytope $\bigcup_\F X^*_{\F^*}$. If ${\mathbb G}=\set{F^0,\ldots,F^{n-1}}$ is a flag over $H$ containing $v$ and $E_1$, then the simplex by $x\sub{(F^{0})^*},\ldots,x\sub{(F^{n-1})^*}$ and $x^\star$ is contained in $K^\circ$, not in the polytope $\bigcup_\F X^*_{\F}$. So 
\begin{align*}
\abs{K^\circ} &\ge V(X^*)+\abs{\conv\set{x^\star,x\sub{(F^{0})^*},\ldots,x\sub{(F^{n-1})^*}}}\\
&= V(X^*)+\eps\de\abs{X^*_{\mathbb G}},
\end{align*}
which implies
\begin{align*}
\abs{K}\abs{K^\circ} &\ge \abs{K}V(X^*)+\eps\abs{K}\abs{X^*_{\mathbb G}}\de\\
&\ge V(X)V(X^*)+\Big(\frac12\cdot\eps\cdot\abs{B_1^n}\cdot\frac{\abs{B_1^n}}{n!2^n}\Big)\de.
\end{align*}
On the other hand, if $K\not\subset(1+\eps^2\de)\brR{\bigcup_\F X_\F}$, then there exist a flag $\F$ and a point $x\in\conv\set{x\sub{F}:F\in\F}$ such that $(1+\eps^2\de)x\in K$. Since it gives $\abs{K}\ge V(X)+\eps^2\de\abs{X_\F}$,
\begin{align*}
\abs{K}\abs{K^\circ} &\ge V(X)\abs{K^\circ}+\eps^2\abs{X_\F}\abs{K^\circ}\de\\
&\ge V(X)V(X^*)+\Big(\frac12\cdot\eps^2\cdot\frac{\abs{B_1^n}}{n!2^n}\cdot\abs{B_1^n}\Big)\de.
\end{align*}
Consequently, we have $\abs{K}\abs{K^\circ} \ge V(X)V(X^*)+c(n)\de$ where $c(n)=\frac{\eps^2\abs{B_1^n}^2}{n!2^{n+1}}=\frac{2^{n-1}\eps^2}{n!^3}$.
\end{proof}

\begin{proof}[\bf Proof of Theorem \ref{thm:delta_gaps}]
Let $H$ be a standard Hanner polytope in $\R^n$, and $K$ a symmetric convex body in $\R^n$ with $\dH(K,H)=\de$ for small $\de>0$. As mentioned in Section \ref{general_construction}, note that $V(X)$ can be viewed as the volume of a star-shaped, not necessary convex, polytope $P=\bigcup_\F X_\F$.. First, consider the case that the (Hausdorff) distance between $K$ and $P$ is comparable to $\de$. Then, there exist a flag $\F$ and a point $x\in\conv\set{x\sub{F}:F\in\F}$ with $(1+c\de)x\in K$. The convex hull of $(1+c\de)x$ and all $x\sub{F}$ for $F\in\F$ is contained in $K\setminus{\rm int}(P)$, and its volume is $c\de$ times the volume of $X_\F$ because the distance between $x$ and $(1+c\de)x$ is $c\de$ times the distance between $x$ and the origin. Thus, in this case
\begin{align*}
\abs{K}\abs{K^\circ} &\ge \brR{V(X)+c\de\abs{X_\F}}\,V(X^*) \\
&\ge V(X)V(X^*)+\frac{c\abs{B_1^n}^2}{n!2^n}\,\de.
\end{align*}
To prove the other case $\dH(K,P)=o(\de)$, we apply Propositions \ref{prop:step1},\ref{prop:step2},\ref{prop:step3}. So, there exists a polytope $\tilde{K}$ with $\dBM(\tilde{K},\conv(P))=1+O(\de^2)$ such that $$|\tilde{K}||\tilde{K}^\circ|\ge V(\tilde{X})V(\tilde{X}^*)+c(n)\de,$$
where $\tilde{X}=(x\sub{F})$, $\tilde{X}^*=(x\sub{F^*})$ are the sets of the $x\sub{F}$-points obtained from $\tilde{K}$, $\tilde{K}^\circ$ as in Definition \ref{def:AF}.
Moreover, we get $\dBM(\tilde{K},K)\le\dBM(\tilde{K},P)\cdot\dBM(K,P)=1+o(\de)$ in this case, which complete the proof.
\end{proof}

\begin{proof}[\bf Proof of Main theorem]
After taking a linear transformation and a small $o(\de)$-perturbation on $K$, by Theorem \ref{thm:delta_gaps} we have $$\abs{K}\abs{K^\circ} \ge  V(X)V(X^*)+c(n)\de.$$ Propositions \ref{prop:vol_difference} and Corollary \ref{cor:differential_zero} imply $$|V(X)-V(Y)|=O(\de^2)=|V(X^*)-V(Y^*)|.$$ Thus
\begin{align*}
\abs{K}\abs{K^\circ} &\ge  V(X)V(X^*)+c(n)\de\\
 &\ge V(Y)V(Y^*)+c'(n)\de.
\end{align*}
Finally, since $ V(Y)V(Y^*)\ge V(C)V(C^*)=|H||H^\circ|$ by Propositions \ref{prop:split_equal_volumes} and \ref{prop:equal_volumes}, we have $$\PP(K)\ge\PP(H)+c(n)\de.$$
\end{proof}





\noindent J. Kim: Department of Mathematics, Kent State
University, Kent, OH 44242, USA.\newline Email:  jkim@math.kent.edu

\end{document}